%% file: boldini-cc_cv_wa.tex
\documentclass[a4paper]{amsart}

\usepackage{amssymb}

\usepackage{enumerate}
\usepackage{subfigure}
\usepackage[usenames,dvipsnames]{color}
\usepackage[dvips]{graphicx}
\usepackage[dvipsnames]{pstricks}
\usepackage{pstricks-add}
\usepackage{float}
\usepackage{microtype}

\newtheorem{theorem}{Theorem}[section]
\newtheorem{main}[theorem]{Main Theorem}
\newtheorem{lemma}[theorem]{Lemma}
\newtheorem{proposition}[theorem]{Proposition}
\newtheorem{corollary}[theorem]{Corollary}

\theoremstyle{definition}
\newtheorem{definition}[theorem]{Definition}
\newtheorem{notation}[theorem]{Notation}
\newtheorem{reminder}[theorem]{Reminder}
\newtheorem{example}[theorem]{Example}
\newtheorem{experiment}[theorem]{Experiment}

\newtheorem{question}[theorem]{Question}

\newtheorem{remark}[theorem]{Remark}

\numberwithin{equation}{section}

\newcommand{\igb}{\gamma}
\newcommand{\sw}{{\nu + s \omega}}
\newcommand{\concept}[1]{\emph{#1}}

\newcommand{\filt}[1]{\mathcal{#1}}
\newcommand{\Fi}{\mathrm{F}}
\newcommand{\Gr}{\mathrm{G}}
\newcommand{\grad}[1]{\Gr\mathcal{#1}}
\newcommand{\und}{\,\mathop{\wedge}\,}

\newcommand{\dann}{\Rightarrow}

\newcommand{\epi}{\twoheadrightarrow}
\newcommand{\mono}{\rightarrowtail}

\renewcommand{\vec}{\mathop{\vert}}
\renewcommand{\mod}{ / } 
\renewcommand{\epsilon}{\varepsilon}
\renewcommand{\varsigma}{\sigma}
\renewcommand{\limsup}{\overline{\lim}}
\renewcommand{\Im}{\Img}

\newcommand{\NN}{\mathbb{N}}
\newcommand{\ZZ}{\mathbb{Z}}

\newcommand{\RR}{\mathbb{R}}
\newcommand{\CC}{\mathbb{C}}
\newcommand{\MMM}{M}
\newcommand{\NNN}{N}
\newcommand{\nth}{n^{\text{th}}}
\newcommand{\ith}{i^{\text{th}}}
\renewcommand{\th}[1]{{#1}^{\text{th}}}
\newcommand{\rad}[1]{{\surd\hspace{1pt}#1}} 
\newcommand{\dd}{\partial}
\newcommand{\csi}{\xi}
\newcommand{\wo}{\smallsetminus}

\DeclareMathOperator{\NO}{\mathcal{O}}
\DeclareMathOperator{\lt}{lm}
\DeclareMathOperator{\LT}{LM}

\DeclareMathOperator{\Img}{Im}

\DeclareMathOperator{\Spec}{Spec}
\DeclareMathOperator{\Var}{Var}

\DeclareMathOperator{\supp}{supp}
\DeclareMathOperator{\End}{End}

\DeclareMathOperator{\Ext}{Ext}

\DeclareMathOperator{\Kdim}{Kdim}
\DeclareMathOperator{\GKdim}{GKdim}
\DeclareMathOperator{\len}{len}
\DeclareMathOperator{\lev}{\deg}

\begin{document}

\title{Critical Cones of Characteristic Varieties}


\author{Roberto Boldini}
\address{%
Universit\"at Z\"urich,
Institut f\"ur Mathematik,
Winterthurerstr.~190, 
CH-8057 Z\"urich} 
\email{roberto.boldini@math.uzh.ch}
\thanks{My gratitude to Prof.~em.~Markus Brodmann and Prof.~Joseph Ayoub, University of Zurich}



\date{on July $\text{21}^\text{st}$, 2010}


\commby{Prof.~Robert Guralnick}

\relpenalty 9999 
\binoppenalty 10000
\hyphenpenalty 0

\begin{abstract}
Let $M$ be a left module over a Weyl algebra in characteristic zero.
Given natural weight vectors $\nu$ and $\omega$,
we show that the characteristic varieties arising from filtrations
with weight vector $\nu+s\omega$ stabilize to a certain variety determined
by $M$, $\nu$, $\omega$
as soon as the natural number $s$ grows beyond a bound which depends only on $M$ and $\nu$
but not on $\omega$.

As a consequence, in the notable case when $\nu$ is the standard weight vector,
these characteristic varieties deform to the critical cone of the $\omega$-characteristic
variety of $M$ as soon as $s$ grows beyond an invariant of $M$.

As an application, we give a new, easy, non-homological proof of a classical result, namely,
that the $\omega$-characteristic varieties of $M$ all have the same Krull dimension.

The set of all $\omega$-characteristic varieties of $M$ is finite.
We provide an upper bound for its cardinality in terms of supports of
universal Gr\"obner bases in the case when $M$ is cyclic.
By the above stability result we conjecture a second upper bound
in terms of total degrees of universal Gr\"obner bases and of Fibonacci numbers
in the case when $M$ is cyclic over the first Weyl algebra.
\end{abstract}

\maketitle

\input{boldini-cc_cv_wa-0.tex}
\input{boldini-cc_cv_wa-1.tex}

\input{boldini-cc_cv_wa-2.tex}

\input{boldini-cc_cv_wa-3.tex}

\input{boldini-cc_cv_wa-4.tex}

\input{boldini-cc_cv_wa-5.tex}

\input{boldini-cc_cv_wa-6.tex}
\bibliographystyle{amsplain}

\end{document}

%% file: boldini-cc_cv_wa-0.tex
\section*{Introduction}

\noindent
Let $n\in\NN$, let $W$ be the $\nth$ Weyl algebra over a field $K$ of characteristic $0$,
and let $\Omega=\{\omega\in\NN_0^{2n}\mid\omega_i+\omega_{i+n}>0 \text{ for } 1\le i\le n\}$.
For each $\omega\in\Omega$
consider the $\omega$-degree filtration $\Fi^\omega W=(\Fi_i^\omega W)_{i\in\ZZ}$ of $W$
and any  good $\Fi^\omega W$-filtration $\Fi^\omega M=(\Fi_i^\omega M)_{i\in\ZZ}$ of
a left $W$-module $M$. 
We construct 
$\Gr^\omega W=\bigoplus_{i\in\ZZ}\Fi_i^\omega W\mod \Fi_{i-1}^\omega W$
and 
$\Gr^\omega M=\bigoplus_{i\in\ZZ}\Fi_i^\omega M\mod \Fi_{i-1}^\omega M$.
Then 
$\Gr^\omega W$ is a ring canonically isomorphic to the commutative
polynomial ring $K[X,Y]$ in the indeterminates $X_1,\ldots,X_n$ and $Y_1,\ldots,Y_n$,
and $\Gr^\omega M$ is a finitely generated $K[X,Y]$-module.
For a fixed $\omega\in\Omega$, the radical ideal $\rad{(0:\Gr^\omega M)}$ of $K[X,Y]$
is independent of the choice of a good $\Fi^\omega W$-filtration $\Fi^\omega M$ of $M$.
So we may define the $\omega$-characteristic variety of $M$ as
the closed subset $\mathcal{V}^\omega(M)={\Var(0:\Gr^\omega M)}$
of $\Spec(K[X,Y])$.

Similarly, we consider the $\nu$-degree filtrations $\Fi^\nu K[X,Y]$ of $K[X,Y]$,
$\nu\in\NN_0^{2n}$, and good $\Fi^\nu K[X,Y]$-filtrations $\Fi^\nu N$
of $K[X,Y]$-modules $N$ and construct the rings $\Gr^\nu K[X,Y]$,
canonically isomorphic to $K[X,Y]$, and the finitely generated $K[X,Y]$-modules $\Gr^\nu N$.
Again, for a fixed $\nu\in\NN_0^{2n}$, the radical ideal $\rad{(0:\Gr^\nu N)}$
of $K[X,Y]$ does not depend on the choice of a good $\Fi^\nu K[X,Y]$-filtration $\Fi^\nu N$
of~$N$.

The main result of this paper is that
for each $\nu\in\NN_0^{2n}$ there exists $s_0\in\NN_0$
such that
for all $s\in\NN$ with $s>s_0$ and all $\omega\in\Omega$
in $K[X,Y]$ it holds
\begin{equation}\tag{A}\label{intro main}
\textstyle{
\rad{(0:\Gr^\nu\Gr^\omega M)}=\rad{(0:\Gr^{\nu+s\omega}M)}.
}
\end{equation}

Observe that $s_0$ does not depend on $\omega$.
We can choose the lowest such $s_0$ in $\NN_0$, denoted $\kappa_\nu(M)$.
If $L$ is a left ideal of~$W$, we give an upper bound for $\kappa_\nu(W/L)$ in
terms of total degrees of elements of universal Gr\"obner bases of $L$, more precisely,
\begin{equation}\tag{B}
\textstyle{
\kappa_\nu(W/L)\le\gamma_\nu(L),
}
\end{equation}
where
\begin{equation*}
\textstyle{
\gamma_\nu(L)=\inf_U{}\sup_{u\in U}{}\deg^\nu(u),
}
\end{equation*}
the infimum being taken over all universal Gr\"obner bases $U$ of $L$.

A case with evident geometrical meaning is when $\nu=(1)=(1,\ldots,1)\in\NN_0^{2n}$.
The equality \eqref{intro main} says that the ``affine deformations''
$\mathcal{V}^{(1)+s\omega}(M)$ of $\mathcal{V}^\omega(M)$
stabilize for large $s$ to the critical cone 
$\mathcal{C}^\omega(M)=\Var(0:\Gr^{(1)}\Gr^\omega M)$
of $\mathcal{V}^\omega(M)$.
Thus the minimal limit beyond which this occurs, namely,
$\kappa(M)=\smash{\kappa_{(1)}(M)}$,
is ---surprisingly--- an invariant of $M$.
Upper bounds for the greatest total degree of Gr\"obner bases and 
of reduced Gr\"obner bases of a left ideal $L$ of $W$ are given in \cite{Ley} in terms of
greatest total degrees of systems of generators of $L$, and hence, combining both results,
we obtain an upper bound for $\kappa(W\mod L)$ also in such terms.

The critical cone $C$ of an affine variety $V\subseteq\mathbb{A}^{r}$
over an algebraically closed field $F$ is the cone with vertex at the origin
$O\in\mathbb{A}^{r}$ tangent to $V$ at infinity.
In other words, $C$ consists of all lines through $O$ along whose directions
$V$ goes to infinity.
To construct $C$, we choose an injection $\iota:\mathbb{A}^{r}_{}\mono\mathbb{P}^{r}$ of
$\mathbb{A}^{r}$ into the projective space $\mathbb{P}^{r}$ over $F$
and put $$C=\textstyle{\iota^{-1}(\bigcup_{P\in\overline{\iota(V)}\wo\iota(V)}\ell_{P})},$$
where $\smash{\overline{\iota(V)}}$ is the projective closure of $\iota(V)$ in $\mathbb{P}^{r}$
and $\ell_{P}$ is the projective line through the points $\iota(O)$ and $P$.
One has that $C$ does not depend on the choice of~$\iota$.
Algebraically, if $I$ is any ideal of $F[Z_1,\ldots,Z_{r}]$ that defines $V$,
then $C$ is defined by the ideal $J$ generated by the homogeneous components
of greatest total degree of the polynomials in $I$, that is,
$J$ is the leading form ideal of $I$ by total degree.
Again, $C$ does not depend on the choice of $I$.

As a further consequence of the equality \eqref{intro main},
we are able to give an easy proof that
$\Kdim_{K[X,Y]}\Gr^\omega M=\GKdim_W M$
for all $\omega\in\Omega$.
Thus,
without having to appeal to sophisticated homological methods as in classical proofs,
we have shown in particular that the characteristic varieties $\mathcal{V}^\omega(M)$,
$\omega\in\Omega$, all have the same Krull dimension.
The key point is that \eqref{intro main} allows in some sense to pass from non-finite to
finite filtrations, and 
Gelfand--Kirillov dimension behaves well with  finite  discrete  filtrations:
$\GKdim_{\Gr^\omega W}\Gr^\omega M=\GKdim_{W}M$ whenever $\Fi^\omega M$ is
finite and discrete.
The second point is that Gelfand--Kirillov dimension and Krull dimension agree in the category 
of finitely generated modules over any fixed finitely generated algebra over a field.

Fixed a left ideal $L$ of $W$, we give an upper bound for the number $\chi(L)$ of distinct
ideals $\Gr^\omega L$, $\omega\in\Omega$, and hence of distinct
$\omega$-characteristic varieties of $W\mod L$, namely,
\begin{equation}\tag{C}
\textstyle{
\chi(L)\le\smash{\inf_{U}^{}\prod_{u\in U}\sum_{0\le k\le\#\supp(u)}\binom{\#\supp(u)}{k}},
}
\end{equation}
the infimum being taken over all universal Gr\"obner bases of $L$.
The equality \eqref{intro main} let us conjecture a second upper bound in the case
when $W$ is the $1^\text{st}$ Weyl algebra, namely,
\begin{equation}\tag{D}\label{D}
\textstyle{
\chi(L)\le\smash{2^{1+\gamma(L)}+1},
}
\end{equation}
where $\gamma(L)=\gamma_{(1)}(L)$.
As mentioned afore,
by~\cite{Ley} it follows an upper bound for $\gamma(L)$ in terms of total degrees of
generators of~$L$.

In Section~\ref{sec1} we recall some known facts about filtered rings and modules as well as
their associated graded rings and modules.

In Section~\ref{sec2} we introduce Weyl algebras and state some of their basic properties,
which are a generalization of results that can be found for instance in \cite{Cou}.
The proofs remain very similar, and we omit them here.

Section~\ref{sec3} is about Gr\"obner bases in Weyl algebras.
Here, too, we recall known facts, important in the next section,
in particular the existence of universal Gr\"obner bases for left ideals,
and a very tight relation between the Gr\"obner bases of $\omega$-filtered left ideals 
and the Gr\"obner bases of their associated graded ideals.

In Section~\ref{sec4} we define $\omega$-characteristic varieties of a left $W$-module $M$
as some particular affine spectra, and not as algebraic zero sets,
as it is usual, for there is no reason here to work only over algebraically closed fields.
Then we prove our main result~\eqref{intro main} about the defining annihilators of such
varieties.

In section~\ref{sec5} we apply \eqref{intro main} to give an easy proof of the known result
that the $\omega$-characteristic
varieties of $M$ all have the same Gelfand--Kirillov and Krull dimension as $\omega$ varies
in $\Omega$, namely, equal to the Gelfand--Kirillov dimension of~$M$.

Finally in Section~\ref{sec6} we perform a computer experiment in order to try to classify
the $\omega$-characteristic varieties of $M$ in the case when $M=W\mod L$ for a left ideal
$L$ of $W$.
This experiment let us conjecture an upper bound for their number, namely~\eqref{D}.

%% file: boldini-cc_cv_wa-1.tex
\section{Filtrations and Gradings}\label{sec1}

\noindent
In this section we give a small review on filtered rings and modules and their associated
graded objects.
Most of this material can be found or inferred from the books of Constantin N\u ast\u asescu,
Freddy van Oystaeyen, and Huishi Li, among which we particularly appreciate~\cite{H-O}.
Besides giving the very short proof of \ref{1.21}, we provide a proof of \ref{1.26}
and \ref{1.27}, too,
which we did not find in the literature.

\begin{definition}
A \concept{filtration} $\filt{R}$ of a ring $R$ is a family $(\Fi_i\filt{R})_{i\in\ZZ}$
of additive subgroups $\Fi_i\filt{R}$ of $R$
enjoying the properties:
$\text{(a)}~{R=\textstyle{\bigcup_{i\in\ZZ}}\Fi_i\filt{R}}$,
$\text{(b)}~{\Fi_{i-1}\filt{R}\subseteq\Fi_{i}\filt{R}}$,
$\text{(c)\,}~{r\in\Fi_i\filt{R} \und s\in\Fi_j\filt{R} \dann rs\in\Fi_{i+j}\filt{R}}$,
$\text{(d)}~{i<0\dann\Fi_i\filt{R}=0}$,
$\text{(e)}~{1\in\Fi_0\filt{R}}$,
so that $\Fi_0\filt{R}$ is a subring of $R$ and each $\Fi_i\filt{R}$ is a left
$\Fi_0\filt{R}$-submodule of $R$.

If the ring $R$ is provided with a filtration $\filt{R}$,
we say that the ordered pair ${(R,\filt{R})}$ is a \concept{filtered ring}.

Let ${(R,\filt{R})}$ and ${(S,\filt{S})}$ be filtered rings.
A \concept{homomorphism of $(R,\filt{R})$ in $(S,\filt{S})$} is a ring homomorphism  
$\phi$ of $R$ in $S$ such that $\phi(\Fi_i\filt{R})\subseteq\Fi_i\filt{S}$.
\end{definition}

\begin{definition}
Let ${(R,\filt{R})}$ be a filtered ring.
An \concept{$\filt{R}$-filtration} $\filt{M}$ of a left $R$-module $M$ is a family
$(\Fi_i\filt{M})_{i\in\ZZ}$ of additive subgroups $\Fi_i\filt{M}$ of $M$
with the properties:
$\text{(a)}~{M=\textstyle{\bigcup_{i\in\ZZ}}\Fi_i\filt{M}}$,
$\text{(b)}~{\Fi_{i-1}\filt{M}\subseteq\Fi_{i}\filt{M}}$,
$\text{(c)\,}~{r\in\Fi_i\filt{R} \und m\in\Fi_j\filt{M} \dann rm\in\Fi_{i+j}\filt{M}}$,
so that each $\Fi_i\filt{M}$ is a left $\Fi_0\filt{R}$-submodule of $M$.

If the left $R$-module $M$ is provided with an $\filt{R}$-filtration $\filt{M}$,
we say that the ordered pair ${(M,\filt{M})}$ is an \concept{$\filt{R}$-filtered left
$R$-module} or simply a \concept{left $(R,\filt{R})$-module}.
Observe that a filtered ring is also a filtered left module over itself.

Let ${(M,\filt{M})}$ and ${(N,\filt{N})}$ be left ${(R,\filt{R})}$-modules.
An \concept{${(R,\filt{R})}$-homomorphism of $(M,\filt{M})$ in $(N,\filt{N})$}
is a left $R$-module homomorphism
$\phi$ of $M$ in $N$ such that $\phi(\Fi_i\filt{M})\subseteq\Fi_i\filt{N}$.
%
\end{definition}

\begin{definition}\label{filtrazione indotta}
Let ${(R,\filt{R})}$ be a filtered ring and
${(M,\filt{M})}$ be a left $(R,\filt{R})$-module.
Let $N$ be a left $R$-submodule of $M$.
There exist canonically \concept{induced $\filt{R}$-filtrations}
$\filt{N}=(\Fi_i\filt{M}\cap N)_{i\in\ZZ}$
of $N$
and
$\filt{M}\mod\filt{N}=(\Fi_i\filt{M}+N\mod N)_{i\in\ZZ}$ of $M\mod N$.
In this situation we call $(N,\filt{N})$ a \concept{submodule of $(M,\filt{M})$}
and $(M/N,\filt{M}/\filt{N})$ a \concept{quotient module of $(M,\filt{M})$}.
Similarly, if $I$ is a left ideal of $R$
and $\filt{I}$ is the induced $\filt{R}$-filtration of $I$,
we say that $(I,\filt{I})$ is a \concept{left ideal of $(R,\filt{R})$}.
\end{definition}

\begin{definition}\label{associato}
Let $(R,\filt{R})$ be a filtered ring.
The \concept{associated graded ring $\grad{R}$ of $R$ with respect to $\filt{R}$} is
the commutative group
$\bigoplus_{i\in\ZZ}\Fi_{i}\filt{R}\mod\Fi_{i-1}\filt{R}$
provided with a multiplication given by
$(r_i+\Fi_{i-1}\filt{R})_{i\in\ZZ}\;(s_j+\Fi_{j-1}\filt{R})_{j\in\ZZ}
=(\sum_{i+j=k}r_is_j+\Fi_{k-1}\filt{R})_{k\in\ZZ}$,
which indeed turns
$\grad{R}$ into a ring.

Let $(M,\filt{M})$ be a left $(R,\filt{R})$-module.
The \concept{associated graded left $\grad{R}$-module $\grad{M}$ of $M$ with respect to
$\filt{M}$} is
the commutative group
$\bigoplus_{i\in\ZZ}\Fi_{i}\filt{M}\mod\Fi_{i-1}\filt{M}$
with a $\grad{R}$-action defined by
$(r_i+\Fi_{i-1}\filt{R})_{i\in\ZZ}\;(m_j+\Fi_{j-1}\filt{M})_{j\in\ZZ}
=(\sum_{i+j=k}r_im_j+\Fi_{k-1}\filt{M})_{k\in\ZZ}$,
which indeed turns
$\grad{M}$ into a left $\grad{R}$-module.

$\grad{R}$ is precisely the associated graded left $\grad{R}$-module of $R$ with respect to 
$\filt{R}$.
We denote the $\ith$ homogeneous component $\Fi_i\filt{M}\mod\Fi_{i-1}\filt{M}$ of
$\grad{M}$ by $\Gr_i\filt{M}$.
Then $\Gr_0\filt{R}$ is a subring of $\grad{R}$ and  each $\Gr_{i}\filt{M}$ is a 
left $\Gr_0\filt{R}$-submodule of $\grad{M}$.
\end{definition}

\begin{remark}\label{filtrazione stretta}\label{sequenza esatta}
Let $(R,\filt{R})$ be a filtered ring, $(X,\filt{X})$ and $(Y,\filt{Y})$ be left
$(R,\filt{R})$-modules, and $\phi$ be a homomorphism of $(X,\filt{X})$ in $(Y,\filt{Y})$.
We have canonical  $\Fi_0\filt{R}$-module homomorphisms
$\Fi_i\filt{X}\mod\Fi_{i-1}\filt{X}\rightarrow\Fi_i\filt{Y}\mod\Fi_{i-1}\filt{Y}$
whose direct sum is a graded left $\grad{R}$-module homomorphism
$\grad{X}\rightarrow\grad{Y}$.

If $\smash{(N,\filt{N}) \mono (M,\filt{M}) \epi (P,\filt{P})}$ is a
\concept{strict exact sequence of $(R,\filt{R})$-modules}, that is,
$\smash{N\overset{\nu}{\mono}M\overset{\pi}{\epi}P}$ is an exact sequence of $R$-modules with 
$\nu(\Fi_i\filt{N})=\Fi_i\filt{M}\cap\Im(\nu)$ and
$\pi(\Fi_i\filt{M})=\Fi_i\filt{P}\cap\Im(\pi)$,
then there is an exact sequence $\smash{\grad{N} \mono \grad{M}
 \epi \grad{P}}$ of graded left $\grad{R}$-modules.
 
In particular, if $(N,\filt{N})$ is a submodule of $(M,\filt{M})$
and $(M/N,\filt{M}/\filt{N})$ is a quotient module of $(M,\filt{M})$,
then we obtain an exact sequence
$\grad{N}\mono\grad{M}\epi\grad{M/N}$,
so that $\grad{M/N}\cong\grad{M}\mod\grad{N}$ as graded left $\grad{R}$-modules.
\end{remark}

\begin{remark}\label{monotonia stretta}
Let $(R,\filt{R})$ be a filtered ring, $(M,\filt{M})$ be a left $(R,\filt{R})$-module,
and $(N,\filt{N})$ be a submodule of $(M,\filt{M})$.
By \ref{filtrazione stretta} we may write ${\grad{N}\subseteq\grad{M}}$.

Assume that $N\subsetneq M$.
Then the set $I=\{i\in\ZZ\mid\Fi_i\filt{M}\nsubseteq N\}$ is non-empty.
Assume further that the $\filt{R}$-filtration $\filt{M}$ is \concept{discrete},
that is, $\Fi_i\filt{M}=0$ for ${i\ll 0}$.
Then $I$ admits a unique least element $i_0$. 
Suppose that $\grad{N}=\grad{M}$.
Then $\grad{M\mod N}\cong\grad{M}\mod\grad{N}=0$,
so $(\Fi_i\filt{M}+N)\mod(\Fi_{i-1}\filt{M}+N)\cong\Gr_i\filt{M\mod N}=0$ for all $i\in\ZZ$,
hence $\Fi_i\filt{M}\subseteq\Fi_i\filt{M}+N=\Fi_{i-1}\filt{M}+N$ for all $i\in\ZZ$.
In particular
$\Fi_{i_0}\filt{M}\subseteq\Fi_{i_0-1}\filt{M}+N\subseteq N+N=N$,
thus $i_0\notin I$, a contradiction.

Therefore, under the assumption that $\filt{M}$ is discrete, we have the implication
$N\subsetneq M\dann\grad{N}\subsetneq\grad{M}$,
the property of \concept{strict monotony} of  $\Gr$ for discrete filtrations.
%
\end{remark}


\begin{remark}\label{annullatore}
Let $(R,\filt{R})$ be a filtered ring.
Assume that $\filt{R}$ is \concept{commutative}, that is,
$r\in\Fi_i\filt{R} \und s\in\Fi_j{R} \dann rs-sr\in\Fi_{i+j-1}\filt{R}$.
Then the ring $\grad{R}$ is commutative.
In this situation let $(I,\filt{I})$ be a left ideal of $(R,\filt{R})$
and consider the quotient module $(R/I,\filt{R/I})$ of $(R,\filt{R})$.
Then $\grad{I}=(0:\grad{R/I})$ as ideals of $\grad{R}$ by \ref{filtrazione stretta}.
\end{remark}

\begin{definition}\label{mappa simbolica}
Let $(R,\filt{R})$ be a filtered ring and let $(M,\filt{M})$ be a left $(R,\filt{R})$-module.
We define the \concept{$\filt{M}$-degree} function 
${\lev^\filt{M}:M\rightarrow\ZZ\cup\{-\infty\}}$ by
$\lev^\filt{M}(m)={\inf{}\{ i\in\ZZ \mid m\in\Fi_i\filt{M} \}}$ for all $m\in M$.
In particular, $\deg^\filt{M}(0)=-\infty$.
If $(N,\filt{N})$ is a left submodule of $(M,\filt{M})$,
then $\lev^\filt{N}(n)=\lev^\filt{M}(n)$ for all $n\in N$.
Further it holds
$\lev^\filt{M}(m+n)\le\max{}\{\lev^\filt{M}(m),\lev^\filt{M}(n)\}$
and
$\lev^\filt{M}(rm)\le\lev^\filt{R}(r)+\lev^\filt{M}(m)$
for all $r\in R$ and all $m,n\in M$.

We convene that $\Fi_{-\infty}\filt{M}=0$ and $\Gr_{-\infty}\filt{M}=0$.
For each $i\in\ZZ\cup\{-\infty\}$ 
let us consider the left $\Fi_0\filt{R}$-module epimorphism
$\sigma_i^\filt{M}:\Fi_i\filt{M}\rightarrow\Gr_i\filt{M}$
given by
$m\mapsto m+\Fi_{i-1}\filt{M}$.
Now we  define the
\concept{$\filt{M}$-symbol map $\sigma^\filt{M}:M\rightarrow\Gr\filt{M}$ of $M$}
by 
$\smash{m\mapsto\sigma_{d}^\filt{M}(m)}$ where $d=\lev^\filt{M}(m)$.
We call $\sigma^\filt{M}(m)$ the \concept{$\filt{M}$-symbol of $m$}.
If $(N,\filt{N})$ is a left submodule of $(M,\filt{M})$,
then the image of $\sigma^\filt{N}(n)$ in $\grad{M}$ is precisely $\sigma^\filt{M}(n)$.
Moreover, in general, $\sigma^\filt{M}$ is not additive, and $\sigma^\filt{M}$ is
multiplicative precisely when
$\deg^\filt{M}(rm)=\deg^\filt{R}(r)+\deg^\filt{M}(m)$ for all $r\in R$ and all $m\in M$.
\end{definition}

\begin{remark}\label{simbolo 1}
Let $(R,\filt{R})$ be a filtered ring,
$(M,\filt{M})$ be a left $(R,\filt{R})$-module,
and $(N,\filt{N})$ be a submodule of $(M,\filt{M})$.
The image $\sigma^\filt{N}(N)$ consists precisely of all homogeneous elements
of the graded left $\grad{R}$-module $\grad{N}$,
whereas $\sigma^\filt{M}(N)$ consists of the homogeneous elements
of the graded left $\grad{R}$-submodule $\grad{N}$ of $\grad{M}$.

In particular $\grad{N}$ is generated by $\sigma^\filt{N}(N)$ as a left $\grad{R}$-module,
and $\grad{N}$ is generated by $\sigma^\filt{M}(N)$ as a left $\grad{R}$-submodule
of $\grad{M}$,
and for any subset $U$ of $N$ we have that $\sigma^\filt{N}(U)$ generates $\grad{N}$ as a
left $\grad{R}$-module
if and only if $\sigma^\filt{M}(U)$ generates $\grad{N}$ as a left $\grad{R}$-submodule
of $\grad{M}$.
\end{remark}

\begin{proposition}\label{simbolo 2}
Let $(R,\filt{R})$ be a commutatively filtered ring.
Let $I$ be a left ideal of~$R$ and $\filt{I}$ and $\filt{R\mod I}$ be the induced
$\filt{R}$-filtrations of~$I$ and $R\mod I$, respectively.
Then $(0 : \grad{R\mod I})=\grad{I}=\sum_{x\in I}\grad{R}\;\sigma^\filt{R}(x)$ as ideals 
of $\hspace{1pt}\grad{R}$.
\end{proposition}

\begin{proof}
Clear by \ref{annullatore} and \ref{simbolo 1}.
\end{proof}

\begin{remark}\label{machedici}
Let $(R,\filt{R})$ be a filtered ring and $(M,\filt{M})$ be a left $(R,\filt{R})$-module.
If $U$ is a system of generators of $M$ other than $M$, then $\grad{M}$ is not generated by 
$\sigma^\filt{M}(U)$, in general.

For instance consider the commutative polynomial ring  $R=\CC[X]$
provided with the filtration $\mathcal{R}$  given by
$\mathrm{F}_i\mathcal{R}=\{r\in R\mid\deg(r)\le i\}$.
Put $(M,\filt{M})=(R,\filt{R})$.
Obviously $\{X,X+1\}$  is a system of generators of $M$.
Further we have $\mathrm{G}\mathcal{R}\cong\CC[X]$ as rings
and $\mathrm{G}\mathcal{M}\cong\CC[X]$ as $\CC[X]$-modules.
In view of these isomorphisms we can write
$\sigma^\mathcal{M}(X+1)=X=\sigma^\mathcal{M}(X)$.
Thus $\mathrm{G}\mathcal{R}\,\sigma^\mathcal{M}(\{X,X+1\})=\CC[X]\, X\subsetneq\CC[X]$.
%
\end{remark}

\begin{remark}\label{mavala}
The converse of \ref{machedici} is partially true.
If $\filt{M}$ is discrete and $U\subseteq M$ is such that $\sigma^\filt{M}(U)$
generates $\grad{M}$ over $\grad{R}$, then $U$ generates $M$  over $R$.
\end{remark}

\begin{remark}\label{GG}
Let $(R,\filt{R})$ be a filtered ring.
We can provide the graded ring $\grad{R}$ with its
\concept{filtration $\filt{GR}$ induced by the grading}
given by
$\Fi_i\filt{GR}=\bigoplus_{j\le i}\Gr_{j}\filt{R}$.
Then we construct the graded ring $\grad{GR}$ associated to the filtered ring
${(\grad{R},\filt{GR})}$.
Since for each $i$ one has a left module isomorphism
$\Fi_i\filt{R}\cong\Fi_i\filt{GR}$
over the isomorphic rings
$\Fi_0\filt{R}\cong\Fi_0\filt{GR}$,
there exists a graded ring isomorphism
$\grad{R}\cong\grad{GR}$.

In a similar manner, if $(M,\filt{M})$ is a left $(R,\filt{R})$-module,
we find an isomorphism 
${\grad{M}\cong\grad{GM}}$
of graded left modules over the isomorphic graded rings
${\grad{R}\cong\grad{GR}}$,
where $\filt{GM}$ is the filtration of $\grad{M}$ given by
$\Fi_i\filt{GM}=\bigoplus_{j\le i}\Gr_j\filt{M}$.
%
\end{remark}


\begin{definition}\label{def buona}
Let $(R,\filt{R})$ be a filtered ring and $M$ be a left $R$-module.
An $\filt{R}$-filtration $\filt{M}$ of $M$ is \concept{good} if
there exist ${s\in\NN_0}$, ${m_1,\ldots,m_s\in M}$, and ${p_1,\ldots,p_s\in\ZZ}$ such that
for all $i\in\ZZ$ it holds $\smash{\Fi_i\filt{M}=\sum_{j=1}^{s}\Fi_{i-p_j}\filt{R}\;m_j}$.
Since $1\in\Fi_0\filt{R}$, we have then $\smash{m_j\in\Fi_{p_j}\filt{M}}$.
\end{definition}

\begin{remark}\label{good => d}
In the notation of \ref{def buona},
any good $\filt{R}\text{-}$filtration $\filt{M}$ of $M$ is discrete
as $\filt{R}$ is discrete by definition.
\end{remark}

\begin{example}\label{buona}
Let $(R,\filt{R})$ be a filtered ring and $M$ be a finitely generated left $R$-module.
For each finite system of generators $m\in M^{\oplus s}$ of $M$ and
each $p\in\ZZ^{\oplus s}$
there exists a \concept{standard} good $\filt{R}$-filtration $\filt{M}$ of $M$ given by
$\Fi_i\filt{M}=\sum_{j=1}^{s}\Fi_{i-p_j}\filt{R}\;m_j$.
\end{example}

\begin{proposition}\label{good => fg}
Let $(R,\filt{R})$ be a filtered ring
and $(M,\filt{M})$ be a left $(R,\filt{R})$-module.
If the $\filt{R}$-filtration $\filt{M}$ is good,
then the left $\grad{R}$-module $\grad{M}$ is finitely generated.
\end{proposition}

\begin{proof}
See \cite[Lemma~I.5.4(2)]{H-O}.
\end{proof}

\begin{definition}
Let $(R,\filt{R})$ be a filtered ring,
$(M,\filt{M})$ be a left $(R,\filt{R})$-module,
and $(m_k)_{k\in\NN}$ be a sequence of elements $m_k$ of~$M$.

Then $(m_k)_{k\in\NN}$ is said to be an \concept{$\filt{M}$-Cauchy sequence}
if for each $j\in\ZZ$ there exists $n_j\in\NN$
such that for all $k,l\ge n_j$ it holds $m_k-m_l\in\Fi_{j}\filt{M}$.

And $(m_k)_{k\in\NN}$ is said to be \concept{$\filt{M}$-convergent} to $m\in M$ if
for each $j\in\ZZ$ there exists $n_j\in\NN$
such that for all $k\ge n_j$ it holds $m_k-m\in\Fi_{j}\filt{M}$.

If every $\filt{M}$-Cauchy sequence of elements of $M$ is $\filt{M}$-convergent,
then $\filt{M}$ is said to be \concept{complete}.

If $\bigcap_{j\in\ZZ}\Fi_j\filt{M}=\{0\}$,
then $\filt{M}$ is called \concept{separated} or \concept{Hausdorff}.
\end{definition}

\begin{remark}\label{d => c}
Discrete filtrations are complete and, trivially, separated.
So are, in particular, our ring filtrations and any good module filtrations.
\end{remark}

\begin{proposition}\label{dfg => good}
Let $(R,\filt{R})$ be a filtered ring
and $(M,\filt{M})$ be a left $(R,\filt{R})$-module.
If the $\mathcal{R}$-filtration $\mathcal{M}$ is separated and the left $\grad{R}$-module
$\grad{M}$ is finitely generated, then $\mathcal{M}$ is good.
\end{proposition}

\begin{proof}
As $\filt{R}$ is discrete and thus complete, we can appeal to \cite[Theorem~I.5.7]{H-O}.
\end{proof}

\begin{corollary}\label{subfilt}
Let $(R,\filt{R})$ be a filtered ring,
$(M,\filt{M})$ be a left $(R,\filt{R})$-module,
and $(N,\filt{N})$ be a submodule of $(M,\filt{M})$, so that by definition
$\mathcal{N}$ is the $\mathcal{R}$-filtration of $N$ induced by $\mathcal{M}$.
If the ring $\grad{R}$ is left noetherian and the $\filt{R}$-filtration $\filt{M}$ is good,
then $\mathcal{N}$ is good, too.
\end{corollary}

\begin{proof}
By \ref{good => fg}, $\grad{M}$ is left noetherian, and so is $\grad{N}$.
By \ref{good => d}, $\mathcal{M}$ is discrete,
and so is $\mathcal{N}$.
We conclude by \ref{d => c} and~\ref{dfg => good}.
\end{proof}

\begin{remark}\label{quotfilt}
Let $(R,\filt{R})$ be a filtered ring
and $(M,\filt{M})$ be a left $(R,\filt{M})$-module.
Let $N$ be a left $R$-submodule of $M$.
If the $\filt{R}$-filtration~$\filt{M}$ is good
then the induced $\filt{R}$-filtration $\filt{M \mod N}$ of $M\mod N$ is  good.
Indeed, in the notation of \ref{def buona},
one immediately sees that 
$\Fi_i\filt{M\mod N}=\sum_{j=1}^{s}\Fi_{i-p_j}\filt{R}\;(m_j+N)$.
\end{remark}



\begin{definition}
Let $(R,\filt{R})$ be a filtered ring and $M$ be a left $R$-module.
Two $\filt{R}$-filtrations $\filt{M}'$ and $\filt{M}''$ of $M$ are
\concept{equivalent} or \concept{of bounded difference} if 
there exists $r\in\NN$, or equivalently $r\in\ZZ$, such that
$\Fi_{i-r}\filt{M}''\subseteq\Fi_i\filt{M}'\subseteq\Fi_{i+r}\filt{M}''$
for all $i\in\ZZ$.
This defines indeed an equivalence relation among the $\filt{R}$-filtrations of $M$.
\end{definition}

\begin{proposition}\label{equivalente}
Let $(R,\filt{R})$ be a filtered ring and 
$(M,\filt{M}')$ and $(M,\filt{M}'')$ be left $(R,\filt{R})$-modules.
If the $\filt{R}$-filtrations $\filt{M}'$ and $\filt{M}''$ are good,
they are equivalent.
\end{proposition}

\begin{proof}
See \cite[Lemma I.5.3]{H-O}.
\end{proof}

\begin{theorem}\label{radicale}
Let $(R,\filt{R})$ be a filtered ring such that 
the ring filtration
$\filt{R}$ 
is commutative.
Let $(M,\filt{M}')$ and $(M,\filt{M}'')$ be left $(R,\filt{R})$-modules such that
the $\filt{R}$-filtrations
$\filt{M}'$ and $\filt{M}''$
are equivalent.
Then $\rad{(0 : \grad{M}')}=\rad{(0 : \grad{M}'')}$.
\end{theorem}

\begin{proof}
In \cite[Lemma III.4.1.9]{H-O} the claim is stated for good filtrations,
but the authors actually prove it for the more general case of equivalent filtrations.
\end{proof}

\begin{proposition}\label{la2}\label{1.21}
Let $(R,\filt{R})$ be a filtered commutative ring,
$M$ be an $R$-module and  $N$ be an $R$-submodule of $M$.
Providing the annihilators $(0 : M)$, $(0:N)$, $(0:M\mod N)$ in $R$ with the 
respective induced
$\mathcal{R}$-filtrations,
denoted $(0:\mathcal{M})$, $(0:\mathcal{N})$, $(0:\filt{M\mod N})$,
it holds
$\rad{\Gr(0:\filt{M})}=\rad{\Gr(0:\filt{N})}\cap\rad{\Gr(0:\filt{M\mod N})}$
in~$\grad{R}$.
\end{proposition}

\begin{proof}
Let $\overline{x}\in\Gr(0:\mathcal{N}) \cap \Gr(0:\mathcal{M}\mod \mathcal{N})$
be a homogeneous element of degree $i\in\ZZ$.
We find ${u\in\mathrm{F}_i(0:\mathcal{N})=\mathrm{F}_i\mathcal{R}\cap (0:N)}$
and ${v\in\mathrm{F}_i(0:\mathcal{M}\mod \mathcal{N})=\mathrm{F}_i\mathcal{R}\cap (0:M\mod N)}$
with $u+\mathrm{F}_{i-1}\mathcal{R}=\overline{x}=v+\mathrm{F}_{i-1}\mathcal{R}$.
Because $v\in(0:M\mod N)$, it holds $vM\subseteq N$.
Since $u\in(0:N)$, it follows $uvM=0$.
Hence $uv\in(0:M)$.
Since ${u\in\mathrm{F}_i\mathcal{R}}$ and ${v\in\mathrm{F}_i\mathcal{R}}$,
it follows $uv\in\mathrm{F}_{2i}\mathcal{R}\cap(0:M)=\mathrm{F}_{2i}(0:\mathcal{M})$.
So $\overline{x}^2=uv+\mathrm{F}_{2i-1}\filt{R}\in\Gr(0:\mathcal{M})$, 
thus $\overline{x}\in\rad{\Gr(0:\mathcal{M})}$.
We have obtained
${\Gr(0:\mathcal{N})\cap\Gr(0:\mathcal{M}\mod \mathcal{N})
\subseteq
\rad{\Gr(0:\mathcal{M})}}$,
whereas,
on the other hand,
as $(0:M)\subseteq(0:N)\cap(0:M\mod N)$,
it follows from~\ref{monotonia stretta} that
$
\Gr(0:\mathcal{M})\subseteq
\Gr(0:\mathcal{N})\cap\Gr(0:\mathcal{M}\mod \mathcal{N}).
$
Now we pass to the radicals.
\end{proof}



\begin{remark}\label{filtrazione e isomorfismo}\label{1.25}
Let $(R,\filt{R})$ be a filtered ring
and $\phi:M\rightarrow N$ be a an isomorphism of left $R$-modules.
If $\filt{M}$ is an $\filt{R}$-filtration of $M$,
then there exists an $\filt{R}$-filtration $\filt{N}$ of $N$ induced by $\phi$
given by $\Fi_i\filt{N}=\phi(\Fi_i\filt{M})$
such that
there exists a graded $\grad{R}$-isomorphism 
$\smash{\Gr\phi:\grad{M}\rightarrow\grad{N}}$ induced by~$\phi$,
see~\ref{filtrazione stretta}.
Moreover, if $\filt{M}$ is good, then $\filt{N}$ is good, as one checks easily.
\end{remark}

\begin{proposition}\label{annullatore 2}\label{1.26}
Let $R$ be a commutative ring and $\filt{R}$ be a  filtration of $R$ such that
induced $\filt{R}$-filtrations on submodules and quotient modules of $R$
are good.  
Let $M$ be a finitely generated  $R$-module and $\filt{M}$ be an $\mathcal{R}$-filtration 
such that
induced $\filt{R}$-filtrations on submodules and quotient modules of $M$
are good.
Consider the annihilator ${(0:M)}$ of $M$ in $R$ provided with its induced 
$\mathcal{R}$-filtration, which we denote by ${(0:\filt{M})}$.
Then ${\rad{\Gr(0:\filt{M})}=\rad{(0:\grad{M})}}$ as ideals of the commutative ring $\grad{R}$.
\end{proposition}

\begin{proof}
We find $t\in\NN$ such that $M$ is generated by $t$ elements.
%
If $t=1$, there exists an $R$-module isomorphism $\phi:M\rightarrow R\mod I$ for some ideal 
$I$ of $R$.
We furnish the $R$-module $R\mod I$ with the induced $\mathcal{R}$-filtration 
$\mathcal{R}\mod \mathcal{I}$,
good by hypothesis,
and with the $\phi$-induced $\filt{R}$-filtration, denoted $\phi(\filt{M})$,
which is good by~\ref{filtrazione e isomorfismo} since $\filt{M}$ is good by hypothesis.
By~\ref{filtrazione e isomorfismo}, 
 ${{(0:\mathrm{G}\mathcal{M})}}={{(0:\mathrm{G}\phi(\mathcal{M}))}}$.
By~\ref{equivalente} and~\ref{radicale},
\(
{\rad{(0:\mathrm{G}\phi(\mathcal{M}))}}
={\rad{(0:\mathrm{G}\mathcal{R}\mod \mathcal{I})}}
\).
As ${(0 : M)}={(0 : R\mod I)}=I$, 
${(0:\filt{M})}$ is precisely the induced $\filt{R}$-filtration of $I$,
hence by \ref{annullatore} we have
\(
{{(0:\mathrm{G}\mathcal{R}\mod\mathcal{I})}}
=
{{\Gr(0:\filt{M})}}
\).
Thus 
\(
{\rad{(0:\mathrm{G}\mathcal{M})}}
=
{\rad{\Gr(0:\filt{M})}}.
\)

Now let $t>1$.
Assume inductively that the statement holds for all  $R$-modules 
generated by less than $t$ elements.
We find a cyclic  submodule $N$ of $M$ such that $M\mod N$ is generated over $R$ by
$t-1$ elements.
We provide $N$ and $M\mod N$ by the respective  induced filtrations $\mathcal{N}$ and 
$\mathcal{M\mod N}$, which are good,
and provide the  ideals ${(0:N)}$ and ${(0:M\mod N)}$ of $R$ by the respective
induced filtrations, denoted ${(0:\mathcal{N})}$ and ${(0:\mathcal{M\mod N})}$,
which are good by hypothesis.
By the  case with ${t=1}$, we have 
$\rad{\Gr(0:\mathcal{N})}=\rad{(0:\mathrm{G}\mathcal{N})}$.
By the induction hypothesis, we have 
$\rad{\Gr(0:\mathcal{M\mod N})}=\rad{(0:\mathrm{G}\mathcal{M\mod N})}$.
The short exact sequence $N\mono M\epi M\mod N$
of filtered  $R$-modules induces the short exact sequence
${\mathrm{G}\mathcal{N}\mono\mathrm{G}\mathcal{M}\epi\mathrm{G}\mathcal{M\mod N}}$
of graded $\grad{R}$-modules, see~\ref{sequenza esatta}.
Thus 
$
{\rad{(0:\grad{M})}}
=
{\rad{(0:\grad{N})}\cap\rad{(0:\grad{M\mod N})}}
$,
whence
$
{\rad{(0:\grad{M})}}
=
{\rad{\Gr(0:\filt{N})}\cap\rad{\Gr(0:\filt{M\mod N})}}
$.
By~\ref{la2} we get
$
{\rad{(0:\grad{M})}}
=
{\rad{\Gr(0:\filt{M})}}
$.
\end{proof}


\begin{remark}\label{GG2}\label{1.27}
We finish this section with a remark that will be useful later on.
Let $R$ be a commutative ring and $\filt{R}$ be a filtration of $R$, so that 
$\filt{R}$ trivially is commutative.
Let $I$ be an ideal of $R$ and provide $I$ with its induced $\mathcal{R}$-filtration, 
denoted $\mathcal{I}$,
and provide $\rad{I}$ with its induced $\mathcal{R}$-filtration, denoted $\rad{\mathcal{I}}$.
Then $\rad{\Gr\rad{\mathcal{I}}}=\rad{\grad{I}}$.
Indeed
let $\overline{x}\in\Gr\rad{\mathcal{I}}$ be a homogeneous element of degree $i\in\ZZ$.
So $\overline{x}=x+\mathrm{F}_{i-1}\mathcal{R}$ for some
$x\in\mathrm{F}_i\rad{\mathcal{I}}=\mathrm{F}_i\mathcal{R}\cap\rad{I}$.
We find $k\in\NN$ such that $x^k\in I$, and so
$x^k\in\mathrm{F}_{ki}\mathcal{R}\cap I=\mathrm{F}_{ki}\mathcal{I}$,
thus $\overline{x}^k=x^k+\mathrm{F}_{ki-1}\mathcal{R}\in\Gr\mathcal{I}$,
hence $\overline{x}\in\rad{\Gr\mathcal{I}}$.
We have shown that $\Gr\rad{\mathcal{I}}\subseteq\rad{\Gr\mathcal{I}}$.
On the other hand, by~\ref{monotonia stretta}, we have 
$\Gr\mathcal{I}\subseteq\Gr\rad{\mathcal{I}}$.
Passing to the radicals, the claim follows.
\end{remark}

%% file: boldini-cc_cv_wa-2.tex
\section{Weyl Algebras}\label{sec2}

\noindent
In this section let $n\in\NN$ and $K$ be a field of characteristic $0$.
We write $K[X,Y]$ for the commutative polynomial ring $K[X_1,\ldots,X_n,Y_1,\ldots,Y_n]$
and denote  its subring $K[X_1,\ldots,X_n]$ by $K[X]$.

For all $(r,s)\in\NN_0\times\NN_0$ we write $(r \vec s)$ for the vector $\omega\in\NN_0^{2n}$ 
with $\omega_i=r$ and $\omega_{n+i}=s$ for $1\le i\le n$.
For all $\alpha,\beta\in\NN_0^n$ we write $(\alpha\vec\beta)$ for the vector 
$\omega\in\NN_0^{2n}$ with $\omega_i=\alpha_i$ and $\omega_{n+i}=\beta_i$ for $1\le i\le n$.
For all $t\in\NN$ and all $\alpha,\beta\in\NN_0^t$ we denote the sum 
$\sum_{i=1}^t\alpha_i\beta_i$ by $\alpha\cdot\beta$.
For all $i\in\{1,\ldots,n\}$ we put $\epsilon^i=(\delta_{ij})_{j=1}^n\in\NN_0^n$,
where $\delta_{ij}\in\NN_0$ is the Kronecker symbol.

We introduce Weyl algebras over $K$ and state some facts about them.
In doing this, we generalize certain well known results that are proved for instance
in \cite{Cou}; the here missing proofs of \ref{2.4} and \ref{2.9}
are elementary but tedious computations and can be mimicked word by word from~\cite{Cou}.

\begin{definition}\label{weyl}
The \concept{$\nth$ Weyl algebra $W$ over} $K$ is defined as the $K$-subalgebra
$K\langle\csi_1,\ldots,\csi_n,\dd_1,\ldots,\dd_n\rangle$
of ${}\End_K(K[X])$ generated by the $K$-linear endomorphisms
$\csi_1,\ldots,\csi_n$ and $\dd_1,\ldots,\dd_n$ of $K[X]$
given by
${\csi_i(p) = X_i p}$
and
$\smash{\dd_i(p)=\frac{\dd p}{\dd X_i}}$
for all ${p\in K[X]}$.
The generators  satisfy the 
\concept{Heisenberg commutation rules}:
$\text{(a)}~{[\csi_i,\csi_j]=0}$,
$\text{(b)}~{[\dd_i,\dd_j]=0}$,
$\text{(c)}~{[\csi_i,\dd_j]+\delta_{ij}=0}$,
where $\delta_{ij}\in K$ is the Kronecker symbol.
\end{definition}

\begin{remark}\label{basis}
As a $K$-module, $W$ has a canonical basis
$\{\csi^\lambda\dd^\mu\mid (\lambda,\mu)\in\NN_0^n\times\NN_0^n\}$,
see~\cite[Satz~2.7]{Bro} or~\cite[Proposition~1.2.1]{Cou}.
As a consequence, for each $w\in W$ there exists a unique function
$c_w:\NN_0^n\times\NN_0^n\rightarrow K$ of finite \concept{support}
$\supp(w)={\{(\lambda,\mu)\in\NN_0^n\times\NN_0^n\mid c_w(\lambda,\mu)\ne 0\}}$
such that $w={\sum c_w(\lambda,\mu)\csi^\lambda\dd^\mu}$ with the sum
 taken over all ${(\lambda,\mu)\in\supp(w)}$.
We write $c_{\lambda\mu}$ for $c_w(\lambda,\mu)$ 
and say that 
${\sum c_{\lambda\mu}\csi^\lambda\dd^\mu}$ 
is the \concept{canonical form} of $w$.
\end{remark}

\begin{definition}\label{qay}
Let
$\deg^\omega(w)={\sup{}\{\omega\cdot (\lambda\vec\mu) \mid (\lambda,\mu)\in\supp(w)\}}$
for all $\omega\in\NN_0^{2n}$ and all $w\in W$,  
the \concept{$\omega$-degree of $w$} with values in $\ZZ\cup\{-\infty\}$.
\end{definition}

\begin{proposition}\label{degw}\label{2.4}
Let $\omega\in\NN_0^{2n}$ and let $u,v\in W$.
Then one has
{\upshape (a)} $\deg^{\omega}(u+v) \le \max\{\deg^{\omega}(u) , \deg^{\omega}(u)\}$,
{\upshape (b)}~$\deg^\omega([u,v])
\!\le\deg^\omega(u)+\deg^\omega(v)-{\min_{1\le i\le n}}\{\omega_i+\omega_{n+i}\}$,
{\upshape (c)}~$\deg^\omega(uv)=\deg^\omega(u)+\deg^\omega(v)$.
Equality holds in {\upshape (a)} if ${}\deg^\omega(u)\ne\deg^\omega(v)$.
\hfill\qedsymbol
\end{proposition}

\begin{definition}\label{w-filt}
Let $\omega\in\NN_0^{2n}$.
Consider the family $\Fi^\omega W=(\Fi^\omega_i W)_{i\in\ZZ}$ defined by 
$\Fi^\omega_i W=\{w\in W\mid \deg^\omega(w)\le i\}$.
Then $\Fi^\omega W$ is a filtration of $W$ by \ref{degw}.
We denote by $\Gr^\omega W$ the associated graded ring  of $W$ with respect to $\Fi^\omega W$,
and by $\Gr_i^\omega W$ the $\ith$ homogeneous component of $\Gr^\omega W$.

Given any  \concept{$\omega$-filtration} ${\Fi^\omega W}$-filtration
$\Fi^\omega M=(\Fi^\omega_i M)_{i\in\ZZ}$ of a left $W$-module $M$, 
we denote by $\Gr^\omega M$ the associated graded left $\Gr^\omega W$-module
associated to $M$ with respect to $\Fi^\omega M$,
and by $\Gr_i^\omega M$  the $\ith$ homogeneous component of $\Gr^\omega M$.

We write $\sigma^\omega$ for the symbol map $W\rightarrow\Gr^\omega W$, and
$\sigma^\omega_i$ for the $\ith$ symbol map $\Fi^\omega_i W\rightarrow \Gr_i^\omega W$.
Thus $\sigma^\omega(w)=\sigma^\omega_{\deg^\omega(w)}(w)$ for all $w\in W$.
\end{definition}

\begin{definition}
We define
$\Omega=\{\omega\in\NN_0^{2n}\mid \omega_{i}+\omega_{n+i}>0 \text{ whenever } 1\le i\le n\}$,
the \concept{natural polynomial region of $W$}.
\end{definition}

\begin{remark}\label{sigma mult}
Let $\omega\in\Omega$ and $v,w\in W$.
As $\deg^\omega(uv)=\deg^\omega(u)+\deg^\omega(v)$ by \ref{degw},
it holds $\sigma^\omega(uv)=\sigma^\omega(u)\sigma^\omega(v)$.
\end{remark}

\begin{remark}\label{omegacomm}
For all $\omega\in\Omega$ the filtration $\Fi^\omega W$ of $W$ is commutative by \ref{degw},
so that the ring $\Gr^\omega W$ is commutative.
\end{remark}

\noindent
Remarks \ref{sigma mult} and \ref{omegacomm}, the canonical injection
$K\mono\Gr^\omega W$,
and the universal property of commutative polynomial rings imply the following theorem.

\begin{theorem}\label{iso}\label{2.9}
For each $\omega\in\Omega$
one has an isomorphism of commutative $K$-algebras
$\psi^\omega:K[X,Y]\rightarrow\Gr^{\omega} W$,
${\sum_{(\lambda,\mu)\in\NN_0^n\times\NN_0^n}}
c_{\lambda\mu}X^\lambda Y^\mu
\mapsto
\textstyle{\sum_{(\lambda,\mu)\in\NN_0^n\times\NN_0^n}}
c_{\lambda\mu}\sigma^\omega(\csi^\lambda)\sigma^\omega(\dd^\mu)$,
which  is graded if we put
$\deg(X_i)=\omega_i$ and $\deg(Y_i)=\omega_{n+i}$
for all $1\le i\le n$.
\hfill\qedsymbol
\end{theorem}

\begin{remark} 
By \ref{iso}, \ref{mavala}, and \ref{degw},
the Weyl algebras are left noetherian domains.
\end{remark}

\begin{remark}
All what we have defined and said in this section about Weyl algebras can be done and proved
in the same way for the commutative polynomial ring $K[X,Y]$, too.
In this situation we may even drop the hypothesis that the field be of characteristic~$0$
and may consider  whole $\NN_0^{2n}$ instead of $\Omega$.
We shall use a similar notation as introduced above for Weyl algebras,
with one exception:
given any $\nu\in\NN_0^{2n}$, we shall write
$\tau_i^\nu$ for the $\ith$ symbol map $\Fi_i^\nu K[X,Y]\rightarrow\Gr_i^\nu K[X,Y]$
and
$\tau^\nu$ for the symbol map $K[X,Y]\rightarrow\Gr^\nu K[X,Y]$,
in order to distinguish them from the symbol maps of the $\nth$ Weyl algebra.
\end{remark}

%% file: boldini-cc_cv_wa-3.tex
\section{Gr\"obner Bases in Weyl Algebras}\label{sec3}

\noindent
In this section
we remind the notion of universal Gr\"obner bases in Weyl algebras and state their existence.
The proof of this fact can be found in \cite{Bol1} and \cite{Bol2}; see also \cite{Wei}.
In \cite{Stu} the same statement is proved for commutative polynomial rings;
a similar proof exists for Weyl algebras.

We keep the notation of the previous section,
and denote
by $\MMM$ the canonical $K$-basis 
$\{X^\lambda Y^\mu\mid(\lambda,\mu)\in\NN_0^n\times\NN_0^n\}$ of~$K[X,Y]$
consisting of the \concept{monomials} $X^\lambda Y^\mu$,
and by $\NNN$ the canonical $K$-basis 
$\{\xi^\lambda\dd^\mu\mid(\lambda,\mu)\in\NN_0^n\times\NN_0^n\}$ of~$W$
consisting of the \concept{normal monomials} $\xi^\lambda\dd^\mu$.

For each $\omega\in\Omega$ we shall tacitly identify the ring $\Gr^\omega W$ with $K[X,Y]$
by means of the $K$-algebra isomorphism $\psi^\omega$ of \ref{iso} and hence
for each left ideal $L$  consider $\Gr^\omega L$ as an ideal of $K[X,Y]$.
Similarly for each $\nu\in\NN_0^{2n}$ we shall identify $\Gr^\nu K[X,Y]$
with $K[X,Y]$ and thus
for each ideal $I$ of $K[X,Y]$  consider $\Gr^\nu I$ as an ideal of $K[X,Y]$.

\begin{definition}
A  \concept{normal ordering},
or  \concept{monomial ordering} in~\cite{Cox},
or \concept{admissible ordering} in~\cite{Wei},
or \concept{term ordering} in~\cite{Sai},
is a total ordering $\preceq$ on $\NN_0^n\times\NN_0^n$ such that
it holds
well-foundedness:
${(0,0)\preceq (\lambda,\mu)}$,
and
compatibility:
${(\lambda,\mu)\preceq (\rho,\sigma)}
\dann
{{(\lambda+\alpha,\mu+\beta)}\preceq {(\rho+\alpha,\sigma+\beta)}}$.
With abuse of notation we write $\xi^\lambda\dd^\mu\preceq\xi^\rho\dd^\sigma$ and
$X^\lambda Y^\mu\preceq X^\rho Y^\sigma$ whenever ${(\lambda,\mu)\preceq (\rho,\sigma)}$.
We denote by $\NO$ the set of all normal orderings.
\end{definition}

\begin{example}\label{lassie}
Lexicographic orderings 
are normal orderings.
\end{example}

\begin{remark}\label{Psi}
There exists a $K$-module isomorphism $\Phi:W\rightarrow K[X,Y]$ which maps
the canonical basis  $\NNN$ of $W$ to the canonical basis $\MMM$ of $K[X,Y]$ by the rule
$\xi^\lambda\dd^\mu\mapsto X^\lambda Y^\mu$.
\end{remark}

\begin{notation}
Let ${\preceq} \in \NO$.
For  $w \in W\wo\{0\}$
we write $\lt_\preceq(w)$ for the greatest normal monomial in the canonical form of $w$
with respect to~${\preceq}$.
We denote $\Phi(\lt_\preceq(w))$ by $\LT_\preceq(w)$.
Given $L\subseteq W$, we often denote by $\LT_\preceq(L)$
the ideal $\sum_{x\in L\wo\{0\}}K[X,Y]\LT_\preceq(x)$ of $K[X,Y]$.
For $p\in K[X,Y]\wo\{0\}$ and $I\subseteq K[X,Y]$
we define $\LT_\preceq(p)$ and $\LT_{\preceq}(I)$ similarly.
\end{notation}

\begin{definition}
Let $L\subseteq W$ be a left ideal and let ${\preceq} \in \NO$.
According to $\text{\cite{Sai}}$, we say that
a finite subset $B$ of $L$
is a \concept{Gr\"obner basis} of $L$ with respect to~$\mathrm{\preceq}$,
or a ${\preceq}$-Gr\"obner basis of $L$,
if it holds ${L=\sum_{b\in B}Wb}$ and
${\LT_\preceq(L)=\sum_{b\in B\wo\{0\}}K[X,Y]\LT_\preceq(b)}$.
Similarly we define a ${\preceq}$-Gr\"obner basis of an ideal $I\subseteq K[X,Y]$,
see~\cite{Cox}.
\end{definition}

\begin{theorem}\label{GB exist}
Let $L\subseteq W$ be a left ideal let and ${\preceq} \in \NO$.
Then $L$ admits a Gr\"obner basis with respect to~${\preceq}$.
\end{theorem}

\begin{proof}
See \cite[Corollary~9.7]{Bol1} or \cite[Theorem~2.15]{Bol2} or \cite[Theorem~1.1.10]{Sai}.
\end{proof}

\begin{definition}
Let $L$ be a left ideal of $W$.
A finite subset $U$ of $L$ is a \concept{universal Gr\"obner basis} of $L$ if $U$ is a
${\preceq}$-Gr\"obner basis of $L$
for each normal ordering~${\preceq}$.
\end{definition}

\begin{theorem}\label{UGB exist}
Each left ideal $L$ of $W$ admits a universal Gr\"obner basis.
\end{theorem}

\begin{proof}
See~\cite[Corollary~10.5 and Example~8.2]{Bol1} or \cite[Theorem~2.22]{Bol2}.
\end{proof}

\begin{remark}\label{indotto1}
For each $\nu\in\NN_0^{2n}$
and each ${\preceq}\in\NO$
there exists  ${\preceq_{\nu}} \in \NO$ defined by
\(
\xi^\lambda\dd^\mu\preceq_\nu\csi^\rho\dd^\sigma \Leftrightarrow
(\lambda\vec\mu)\cdot\nu<(\rho\vec\sigma)\cdot\nu
\,\vee\,((\lambda\vec\mu)\cdot\nu=(\rho\vec\sigma)\cdot\nu
\,\wedge\,(\lambda,\mu)\preceq(\rho,\sigma)).
\)
\end{remark}



\begin{theorem}\label{GB->GS}
Let $\omega\in\Omega$, ${\preceq} \in \NO$, $L\subseteq W$ be a left ideal,
and  $B$ be a ${\preceq}_\omega$-Gr\"obner basis of $L$.
Then $\sigma^\omega(B)$ is a ${\preceq}$-Gr\"obner basis of $\Gr^\omega L$,
thus $\mathrm{G}^\omega L=\langle \sigma^\omega(b)\mid b\in B\rangle$
and $\LT_{\preceq}(\Gr^\omega L)=\langle \LT_{\preceq}(\sigma^\omega(b))\mid b\in B\rangle$
as ideals of $K[X,Y]$.
\end{theorem}


\begin{proof}
See \cite[Theorem~1.1.6(1)]{Sai} or \cite[Propositions~V.7.2 \& II.4.2]{Hui}.
\end{proof}

\begin{remark}\label{GB->GS 2}
Analogously as in \ref{GB->GS},
if $\nu\in\NN_0^{2n}$, ${\preceq} \in \NO$,
$I\subseteq K[X,Y]$ is an ideal,
$B$ is a ${\preceq_\nu}$-Gr\"obner basis of $I$,
then $\tau^\nu(B)$ is a $\preceq$-Gr\"obner basis of $\Gr^\nu I$.
\end{remark}

\begin{corollary}\label{GL finite}
For every left ideal $L$ of $W$
the set $\{\Gr^\omega L\mid\omega\in\Omega\}$ is finite.
Similarly, for every ideal $I$ of $K[X,Y]$
the set $\{\Gr^\nu I\mid\nu\in\NN_0^{2n}\}$ is finite.
\end{corollary}

\begin{proof}
By \ref{UGB exist}, we can find a universal Gr\"obner basis $U\supseteq\{0\}$ of $L$.
By \ref{GB->GS}, $\Gr^\omega L=\langle \sigma^\omega(u)\mid u\in U\rangle$.
So $\smash{\#\{\Gr^\omega L\mid\omega\in\Omega\}\le \prod_{u\in U}\sum_{0\le k\le \#\supp(u)}
\binom{\#\supp(u)}{k}}
<\infty$.
\end{proof}


\begin{remark}
Another proof of \ref{GL finite} by homogenization is in \cite[Theorem~3.6]{ACG}.
\end{remark}

%% file: boldini-cc_cv_wa-4.tex
\section{Characteristic Varieties over Weyl Algebras}\label{sect:tatatataa}\label{sec4}

\noindent
We encounter the notion of characteristic variety and critical cone
and prove our main result,
from which a relation between characteristic varieties and critical cones follows.
We keep the notation of the previous section.

\begin{remark}\label{papparapa}
Fix any $\omega\in\Omega$.
By \ref{iso}, $\mathrm{G}^\omega W\cong K[X,Y]$ as $K\text{-}$algebras.
Let~$M$ be finitely generated left $W$-module.
By \ref{buona} we can provide $M$  with a good ${\omega}\text{-}$filtration ${\Fi^\omega M}$.
By \ref{good => fg}
the $K[X,Y]$-module $\mathrm{G}^\omega M$ is
finitely generated,
and by \ref{omegacomm}, \ref{equivalente}, \ref{radicale} the
ideal $\rad{(0 \mathop{:} \mathrm{G}^\omega M)}$ 
of $K[X,Y]$ is independent of the choice of~${\Fi^\omega M}$.
\end{remark}

\begin{definition}\label{titu}
Let $\omega\in\Omega$ and let $M$ be a finitely generated left $W$-module.
By \ref{papparapa} we may define the
\concept{$\omega\text{-}$characteristic variety $\mathcal{V}^\omega(M)$ of $M$}
as the closed set 
$\Var(\rad{(0 \mathop{:} \mathrm{G}^\omega M)}) 
=\Var(0 \mathop{:} \mathrm{G}^\omega M)$ 
of 
$\Spec(K[X,Y])$.
In particular 
we consider
$\mathcal{V}^{(1\vec 1)}(M)$
and 
$\mathcal{V}^{(0\vec 1)}(M)$,
the characteristic variety of $M$ \concept{by degree}
and
\concept{by order}.

We define the \concept{$\omega$-critical cone $\mathcal{C}^{\omega}(M)$ of $M$}
as  $\Var(\mathrm{G}^{(1\vec 1)}\rad{(0 : \mathrm{G}^{\omega} M)})$,
which is equal to $\Var(\mathrm{G}^{(1\vec 1)}{(0 : \mathrm{G}^{\omega} M)})$
and $\Var(0 : \mathrm{G}^{(1\vec 1)}\mathrm{G}^{\omega} M)$ by
and \ref{GG2} and \ref{annullatore 2},
a closed set of $\Spec(K[X,Y])$.
In particular we consider 
$\mathcal{C}^{(1 \vec 1)}(M)$
and
$\mathcal{C}^{(0 \vec 1)}(M)$,
the critical cone of $M$ \concept{by degree}
and
\concept{by order}.
\end{definition}

\begin{remark}
Let $M$ be a finitely generated left $W$-module and $N$ be a submodule of $M$.
Provided $M$ with a good filtration, 
by \ref{iso} and  by \ref{subfilt} and \ref{quotfilt}
the  induced $\omega$-filtrations of $N$ and $M\mod N$ are good.
Therefore what said in~\ref{papparapa} and~\ref{titu} applies also to $N$ and~$M\mod N$.
\end{remark}

\begin{theorem}\label{qui quo qua}
Given any finitely generated left $W$-module $M$,
there are only finitely many distinct characteristic varieties
$\mathcal{V}^\omega(M)$ for $\omega$ varying in $\Omega$.
\end{theorem}

\begin{proof}
Given a submodule $N$ of $M$,
by \ref{sequenza esatta} one has
$\mathcal{V}^\omega(M)=\mathcal{V}^\omega(N)\cup\mathcal{V}^\omega(M\mod N)$
for all $\omega\in\Omega$.
By induction over the number of generators of $M$,
the claim follows from \ref{GL finite} and \ref{annullatore}.
\end{proof}

\noindent

\begin{lemma}\label{orazio}
Let $w\in W$, $\nu\in\NN_0^{2n}$,  $\omega\in\Omega$.
Let
$l\in\NN_0$ with $l\ge\deg^{\nu}(w)$ in~$W$,
let
$m\in\NN_0$ with $m\ge\deg^{\omega}(w)$ in~$W$,
let
$p\in\NN_0$ with $p\ge\deg^{\nu}(\sigma_{m}^{\omega}(w))$ in $K[X,Y]$.
Then in $K[X,Y]$
for all $s\in\NN$ such that
$s>l-p$
it holds
\(
\smash{\tau^{\nu}_{p}(\sigma^{\omega}_{m}(w))=\sigma^{\sw}_{p+sm}(w)}.
\)
\end{lemma}

\begin{proof}
We write $w$ in canonical form as
$\sum_{(\lambda,\mu)\in\mathbb{S}}c_{\lambda\mu}\xi^\lambda \dd^\mu$,
where $\mathbb{S}=\supp(w)$
and $c_{\lambda\mu}\in {K\wo\{0\}}$.
By definition, we have
$\omega\cdot(\lambda\vec\mu)\le m$ for all $(\lambda,\mu)\in\mathbb{S}$.
Hence
$\sigma^{\omega}_{m}(w)=\sum_{(\lambda,\mu)\in\mathbb{S}_m}c_{\lambda\mu}X^\lambda Y^\mu$,
where $\mathbb{S}_m={\{(\lambda,\mu)\in\mathbb{S}\mid\omega\cdot(\lambda\vec\mu)=m\}}$.
Similarly,
$\nu\cdot(\lambda\vec\mu)\le p$ for all $(\lambda,\mu)\in\mathbb{S}_{m}$.
Hence
$\tau^{\nu}_{p}(\sigma^{\omega}_{m}(w))
=\sum_{(\lambda,\mu)\in\mathbb{S}_{m,p}}c_{\lambda\mu}X^\lambda Y^\mu$,
where $\mathbb{S}_{m,p}={\{(\lambda,\mu)\in\mathbb{S}_m\mid\nu\cdot(\lambda\vec\mu)=p\}}$.

Let $(\lambda,\mu)\in\mathbb{S}$.
As just observed,
$\omega\cdot(\lambda\vec\mu)\le m$,
and moreover
if $\omega\cdot(\lambda\vec\mu)=m$,
then $\nu\cdot(\lambda\vec\mu)\le p$.
Thus we have the following three cases.

If $\omega\cdot(\lambda\vec\mu)=m$ and $\nu\cdot(\lambda\vec\mu)=p$, then
$(\sw)\cdot(\lambda\vec\mu)=\nu\cdot(\lambda\vec\mu)+s\omega\cdot(\lambda\vec\mu)=p+sm$,
hence $\xi^\lambda\dd^\mu\in\Fi^{\sw}_{p+sm}W\wo\Fi^{\sw}_{p+sm-1}W$ 
for all $s\in\NN$.

If $\omega\cdot(\lambda\vec\mu)=m$ and $\nu\cdot(\lambda\vec\mu)<p$, then
$(\sw)\cdot(\lambda\vec\mu)=\nu\cdot(\lambda\vec\mu)+s\omega\cdot(\lambda\vec\mu)<p+sm$,
hence $\xi^\lambda\dd^\mu\in\Fi^{\sw}_{p+sm-1}W$ for all $s\in\NN$.

If $\omega\cdot(\lambda\vec\mu)<m$, then
$(\sw)\cdot(\lambda\vec\mu)=\nu\cdot(\lambda\vec\mu)+s\omega\cdot(\lambda\vec\mu)
\le l+sm-s
<p+sm$
as soon as $s>l-p$,
hence $\xi^\lambda\dd^\mu\in\Fi^{\sw}_{p+sm-1}W$ for all $s\in\NN$ with $s>l-p$.

Therefore, putting $\mathbb{S}'_{m,p}={\{(\lambda,\mu)\in\mathbb{S} \mid
\omega\cdot(\lambda\vec\mu)=m,\;\nu\cdot(\lambda\vec\mu)=p\}}$,
we obtain
$\smash{\sigma^{\sw}_{p+sm}(w)}
=\smash{\sum_{(\lambda,\mu)\in\mathbb{S}'_{m,p}} c_{\lambda\mu} X^\lambda Y^\mu}$
for all $s\in\NN$ with $s>l-p$.
Since $\mathbb{S}_{m,p}=\mathbb{S}'_{m,p}$, we are done.
\end{proof}

\begin{lemma}\label{pluto}
Let $w\in W$,  and let $\nu\in\NN_0^{2n}$ and $\omega\in\Omega$.
Then for all $s\in\NN$ such that 
$s>\deg^{\nu}(w)-\deg^{\nu}(\sigma^\omega(w))$
it holds
$\smash{\deg^{\nu}(\sigma^\omega(w))+s\deg^{\omega}(w)=\deg^{\sw}(w)}$.
\end{lemma}

\begin{proof}
If $w=0$, then the statement holds for all $s\in\NN$.
Hence let $w\ne 0$, and
put $l=\deg^{\nu}(w)$, $m=\deg^{\omega}(w)$ and $\smash{p=\deg^{\nu}(\sigma^\omega_m(w))}$.
Let $s\in\NN$ with $s>l-p$ and put $\smash{d=\deg^{\sw}(w)}$.
As in the proof of \ref{orazio} we obtain
$(\sw)\cdot(\lambda\vec\mu)\le p+sm$ for all $(\lambda,\mu)\in\supp(w)$,
hence $d=\sup{}\{(\sw)\cdot(\lambda\vec\mu)\mid (\lambda,\mu)\in\supp(w)\}\le p+sm$.
If it held $d<p+sm$, then we would have
$\smash{\sigma^{\sw}_{p+sm}(w)=0}$,
whereas
$\smash{\tau^{\nu}_{p}(\sigma^{\omega}_{m}(w))\ne 0}$,
in contradiction to \ref{orazio}.
Hence $p+sm=d$, our claim.
\end{proof}

\begin{lemma}\label{lilly}
Let $w\in W$, and let $\nu\in\NN_0^{2n}$ and $\omega\in\Omega$.
Then
for all $s\in\NN$ such that
$s>\deg^{\nu}(w)-\deg^{\nu}(\sigma^{\omega}(w))$
it holds
$\smash{\tau^{\nu}(\sigma^{\omega}(w))=\sigma^{\sw}(w)}$.
\end{lemma}

\begin{proof}
By \ref{orazio} with
$l=\deg^{\nu}(w)$,
$m=\deg^{\omega}(w)$,
$p=\deg^{\nu}(\sigma^{\omega}_{m}(w))=\deg^{\nu}(\sigma^{\omega}(w))$,
and by \ref{pluto}.
\end{proof}

\noindent
Theorem~\ref{athena} extends a result published in 1971 by Bernstein
as a part of the proof of~\cite[Theorem~3.1]{Ber},
namely that $\smash{\Gr^{(1\vec 1)}\Gr^{(0\vec 1)}L\subseteq\Gr^{(1\vec s)}L}$ for $s\gg 0$.
In greater generality  we prove also the converse inclusion.

\begin{theorem}\label{athena}
Let $L$ be a left ideal of $W$.
For all $\nu\in\NN_0^{2n}$ there exists $s_\nu\in\NN_0$ such that
for all $\omega\in\Omega$ and all $s\in\NN$ with $s>s_\nu$ it holds
\(
\Gr^{\nu}\Gr^\omega L = \Gr^{\sw}L
\)
as ideals of $K[X,Y]$.
\end{theorem}

\begin{proof}
Let $\nu\in\NN_0^{2n}$.
We can choose a universal Gr\"obner basis $U$ of $L$ by \ref{UGB exist},
and we can fix an normal ordering ${\preceq} \in \NO$ by \ref{lassie}.
Thus
$U$ is a $\smash{({\preceq_{\nu}})_{\omega}}$-Gr\"obner basis of $L$ for all $\omega\in\Omega$,
see \ref{indotto1}.

By \ref{GB->GS},
$\sigma^\omega(U)$ is a ${\preceq_{\nu}}$-Gr\"obner basis of $\Gr^\omega L$
for all $\omega\in\Omega$.
Hence, by \ref{GB->GS 2},
$\tau^{\nu}(\sigma^\omega(U))$ is a ${\preceq}$-Gr\"obner basis of $\Gr^{\nu}\Gr^\omega L$
for all $\omega\in\Omega$.
In particular,
$\Gr^{\nu}\Gr^\omega L
=\langle \tau^{\nu}(\sigma^\omega(u))\mid u\in U\rangle$
for all $\omega\in\Omega$.
Putting $s_\nu={\max{}\{\deg^{\nu}(u)\mid u\in U,\,u\ne 0\}}$ if $U\nsubseteq \{0\}$,
and $s_\nu=0$ if $U\subseteq\{0\}$, by \ref{lilly} we get
$\Gr^{\nu}\Gr^\omega L={\langle \sigma^{\sw}(u)\mid u\in U\rangle}$
for all $\omega\in\Omega$ and all $s\in\NN$ with $s>s_\nu$.

On the other hand, $U$ is a Gr\"obner basis of $L$ with respect to ${\preceq_{\sw}}$
for all $\omega\in\Omega$ and all $s\in\NN$.
Therefore, by \ref{GB->GS}, $\sigma^{\sw}(U)$ is a Gr\"obner basis of $\Gr^{\sw}L$ with respect
to~${\preceq}$, whence
${\langle \sigma^{\sw}(u)\mid u\in U\rangle=\Gr^{\sw}L}$,
for all $\omega\in\Omega$ and all $s\in\NN$.
\end{proof}

\begin{main}\label{main}
Let $M$ be a finitely generated left $W$-module.
For all ${\nu\in\NN_0^{2n}}$ there exists $s_\nu\in\NN_0$ with the property that for all
$\omega\in\Omega$ and all $s\in\NN$ with $s>s_\nu$ it holds
$
{\rad{( 0 : \mathrm{G}^{\nu}\mathrm{G}^{\omega}M )}
=\rad{( 0 : \mathrm{G}^{\nu}\mathrm{G}^{\sw}M )}
=\rad{( 0 : \mathrm{G}^{\sw}M )}}
$
as ideals of $K[X,Y]$.
\end{main}

\begin{proof}
We fix any $\nu\in\NN_0^{2n}$.
We find $r\in\NN$ such that $M$ is generated over $R$ by $r$ of its elements.

First, by induction over $r$, we prove the existence of $s_{\nu}\in\NN_0$ such that
for all $\omega\in\Omega$ and all $s\in\NN$ with $s>s_{\nu}$ it holds
$\rad{( 0 : \mathrm{G}^{\nu}\mathrm{G}^{\omega}M )}=\rad{( 0 : \mathrm{G}^{\sw}M )}$.

If $r=1$, then $M\cong W\mod L$ for a left ideal $L$ of $W$.
By \ref{sequenza esatta}, \ref{annullatore}, \ref{athena}
we find $s_{\nu}\in\NN_0$ such that for all $\omega\in\Omega$ and all $s\in\NN$
with $s>s_{\nu}$ it holds
\( 
\rad{(0:\mathrm{G}^{\nu}\mathrm{G}^{\omega}W / L)} 
=\rad{\mathrm{G}^{\nu}\mathrm{G}^{\omega}L} 
=\rad{\mathrm{G}^{\sw}L} \!
=\rad{(0:\mathrm{G}^{\sw}W / L)}.
\) 

If $r>1$,
we find a cyclic submodule $N$ of $M$
such that $P=M\mod N$ is generated by 
$r-1$ elements.
As before, by \ref{athena} we find $s'_{\nu}\in\NN_0$ such that
for all $\omega\in\Omega$ and all $s\in\NN$ with $s>s'_{\nu}$
it holds
\(
\rad{( 0 : \mathrm{G}^{\nu}\mathrm{G}^{\omega}N )}=\rad{( 0 : \mathrm{G}^{\sw}N )}.
\)
By induction we find $s''_{\nu}\in\NN_0$ such that
\(
\rad{( 0 : \mathrm{G}^{\nu}\mathrm{G}^{\omega}P )}=\rad{( 0 : \mathrm{G}^{\sw}P )}
\)
for all $\omega\in\Omega$ and all ${s\in\NN}$ with $s>s''_{\nu}$.
By~\ref{sequenza esatta} we get
\( 
 \rad{( 0 : \mathrm{G}^{\nu}\mathrm{G}^{\omega}M )}
=\rad{( 0 : \mathrm{G}^{\nu}\mathrm{G}^{\omega}N )}
 \cap
 \rad{( 0 : \mathrm{G}^{\nu}\mathrm{G}^{\omega}P )} 
=\rad{( 0 : \mathrm{G}^{\sw}N )}\cap\rad{( 0 : \mathrm{G}^{\sw}P )}
=\rad{( 0 : \mathrm{G}^{\sw}M )}
\) 
for all $\omega\in\Omega$ and all $s\in\NN$ with ${s>s_{\nu}}$,
where $s_{\nu}=\max{}\{s'_{\nu},s''_{\nu}\}$,
so that $s_{\nu}$ is independent of $\omega$.
This completes the induction step.

Now, by \ref{annullatore 2}, \ref{GG2}, \ref{GG}, it follows 
\(
 \rad{(0:\Gr^{\nu}\Gr^{\sw}M)}
=\rad{\Gr^{\nu}\rad{(0:\Gr^{\sw}M)}}
=\rad{\Gr^{\nu}\rad{(0:\Gr^{\nu}\Gr^{\omega}M)}}
=\rad{(0:\Gr^{\nu}\Gr^{\nu}\Gr^{\omega}M)}
=\rad{(0:\Gr^{\nu}\Gr^{\omega}M)}
\)
for all $\omega\in\Omega$ and all $s\in\NN$ with $s>s_{\nu}$.
%
\end{proof}

\begin{corollary}\label{mamma mia}
There exists $s_{(1\vec 1)}\in\NN_0$ such that
for all $\omega\in\Omega$ and all $s\in\NN$ with $s>s_{(1\vec 1)}$ one has
$
{\mathcal{C}^{\omega}(M)
=\mathcal{V}^{(1\vec 1)+s\omega}(M)
=\mathcal{C}^{(1\vec 1)+s\omega}(M)
}.
$
\end{corollary}

\begin{proof}
Clear by \ref{main}.
\end{proof}

\begin{corollary}\label{corcor}
It holds
$\mathcal{C}^{(0 \vec 1)}(M)=\mathcal{V}^{(1 \vec s)}(M)=\mathcal{C}^{(1 \vec s)}(M)$
for $s\gg 0$, 
whereas
$\mathcal{C}^{(1 \vec 1)}(M)=\mathcal{V}^{(1 \vec 1)}(M)$.
\end{corollary}

\begin{proof}
The first statement is clear by \ref{mamma mia}, 
the second follows from \ref{GG}.
\end{proof}

%% file: boldini-cc_cv_wa-5.tex
\section{Application 1: Dimension of Characteristic Varieties}\label{sec5}

\noindent
In this section, as an application of Theorem~\ref{main},
we aim to furnish a new proof of a classical result:
fixed a finitely generated left $W$-module $M$,
the characteristic varieties $\mathcal{V}^{\omega}(M)$, $\omega\in\Omega$,
all have the same Krull dimension.

This is usually proved, as exposed by Ehlers in \cite[Chapter~V]{Bor},
by not trivial homological methods.
It turns out indeed that 
$\Kdim_{K[X,Y]}\Gr^{\omega}M=2n-\mathrm{j}_W(M)$ 
for all $\omega\in\Omega$,
where $\mathrm{j}_W(M)={\inf{}\{i\in\NN_0\mid \Ext_W^i(M,W)\ne 0\}}$.

Bernstein provided in 1971 a proof that $\mathcal{V}^{(1\vec 1)}(M)$ and
$\mathcal{V}^{(0\vec 1)}(M)$ have the same Krull dimension, see \cite[Theorem~3.1]{Ber}.

Our proof descends 
(1)~from the \emph{equality of annihilators} obtained in \ref{main},
which in particular allows to pass in a certain sense from non-finite to finite filtrations,
(2)~from the \emph{preservation of the Gelfand--Kirillov dimension} when passing from finitely 
filtered objects to their associated graded objects, see \ref{GM=M},
and (3)~from the \emph{equality of  Krull and  Gelfand--Kirillov dimension} in the category 
of noetherian modules over a noetherian commutative $K$-algebra, see \ref{Kdim=GKdim 2}.

We begin with some necessary results about the Gelfand--Kirillov dimension that can be found
in \cite{K-L} or \cite{MC-R}.

\begin{reminder}
Let $F$ be a field and $B$ be a finitely generated $F$-algebra.
We  find a \concept{generating space of $B$}, that is, an $F$-module $V$ of finite length 
such that $F\subseteq V$ and $B$ is generated as an $F$-algebra by $V$.
By $V^i$, $i\in\NN_0$, we denote the $F$-module
consisting of all  polynomials
in the (in general not commuting) elements of $V$ with coefficients in $F$
of total degree less than or equal to $i$,
so that in particular $V^0=F$, $V^1=V$, $V^i\subseteq V^{i+1}$,
$B=\smash{\bigcup_{i\in\NN_0}V^i}$.
The \concept{Gelfand--Kirillov dimension of $B$} is defined as
$\GKdim B=\smash{\limsup_{i\rightarrow\infty}\log_i(\len_F V^i)}\in [0,\infty]$,
and it is independent of $V$.
If $A$ is any $F$-algebra, we define $\GKdim A=\sup_B \GKdim B$, where the supremum is
taken over all finitely generated $F$-subalgebras $B$ of $A$.
For finitely generated $F$-algebras the two definitions are easily shown to be equivalent.

Let $N$ be a finitely generated left $B$-module.
We  find a \concept{generating space of $N$}, that is,
an $F$-module $W$ of finite length such that $N$ is generated as a $B$-module by~$W$.
The \concept{Gelfand--Kirillov dimension of $N$} is defined as
$\GKdim_B N=\smash{\limsup_{i\rightarrow\infty}\log_i(\len_F V^i W)}\in[0,\infty]$,
and it is independent of $V$ and of~$W$.
If $M$ is any $A$-module, we define $\GKdim_A M=\sup_B \sup_N \GKdim_B N$, where the suprema
are taken over all finitely generated $F$-subalgebras $B$ of $A$ and all finitely generated
$B$-submodules of~$M$.
For finitely generated modules over finitely generated $F$-algebras
the two definitions are easily shown to be equivalent.
%
\end{reminder}

\begin{reminder}\label{Kdim=GKdim 2}
Let $F$ be a field, $A$ be a finitely generated commutative $F$-algebra,
and $M$ be a finitely generated $A$-module.
Then for the \concept{Krull dimension} $\Kdim_A M$ of $M$, defined as the supremum of
the lengths of chains of prime ideals of the commutative ring ${A\mod (0 : M)}$, it holds
$\Kdim_A M=\GKdim_A M \in \NN_0\cup\{-\infty,\infty\}$.

Indeed, in our hypotheses both dimensions are exact,
see \cite[Theorem~6.14]{K-L} for the Gelfand-Kirillov dimension,
and hence we may assume that $M=A\mod I$ for some ideal $I$.
As both dimensions are preserved when changing the base ring from $A$ to $A\mod I$,
see \cite[Proposition~5.1(c)]{K-L} for the Gelfand--Kirillov dimension,
it is sufficient to compare $\Kdim A\mod I$ to $\GKdim A\mod I$.
As both dimensions are preserved when passing to integral extensions,
see \cite[Proposition~5.5]{K-L} for the Gelfand--Kirillov dimension,
by Emmy Noether's Normalization Lemma
we may replace the finitely generated $F$-algebra $A\mod I$
by the polynomial ring $F[X_1,\ldots,X_d]$, where $d=\Kdim A\mod I$.
By arguments of Linear Algebra, one shows that $\GKdim F[X_1,\ldots,X_d]=d$.
See \cite[Proposition~7.9]{K-L} or \cite[Corollary~1.1.16]{Bel} for more details.

Alternatively, one easily gets
$\GKdim A = {\inf{}\{\alpha\in\RR\mid \len_K V^i\le i^\alpha\text{ for }i\gg 0\}}$,
see \cite[Lemma~2.1]{K-L}.
It follows that $\GKdim A$ is indeed equal to the degree of the Hilbert polynomial of $A$,
which in turn is equal to $\Kdim A$,
and one concludes again
by the exactness of both dimensions
and by changing the base ring.
\end{reminder}



\begin{definition}\label{finfil}
Let $F$ be a field, $A$ be an $F$-algebra, $\mathcal{A}$ be a filtration of $A$,
$M$ be a left $A$-module, and $\mathcal{M}$ be an $\mathcal{A}$-filtration of $M$.
We say that $\mathcal{M}$ is \concept{finite} if $\len_F(\Fi_i\filt{M})<\infty$
for all $i\in\ZZ$.
\end{definition}

\begin{remark}\label{finite and discrete}
In the notation of \ref{finfil}, if $\filt{A}$ is finite
and $M$ is finitely generated and $\filt{M}$ is good, then $\filt{M}$ is finite and discrete.
Indeed, $\filt{M}$ is equivalent to a standard good filtration $\filt{S}$ of $M$,
see \ref{equivalente} and \ref{buona}.
Now, $\filt{S}$ is finite whenever $\filt{A}$ is finite, and $\filt{S}$ is always discrete.
\end{remark}

\begin{lemma}\label{GM=M}
Let $F$ be a field, $A$ be a $K$-algebra, $\filt{A}$ be a filtration of $A$, 
$M$ be a left $A$-module, and $\filt{M}$ be an $\mathcal{A}$-filtration of $M$.
Then $\GKdim_{\mathrm{G}\mathcal{A}} \mathrm{G}\mathcal{M} \le \GKdim_A M$.

Furthermore, if the filtration $\mathcal{A}$ is finite
and is such that the $F$-algebra $\mathrm{G}\mathcal{A}$ is finitely generated,
and if the $\mathcal{A}$-filtration $\mathcal{M}$ is finite and discrete
and is such that the $\mathrm{G}\mathcal{A}$-module $\mathrm{G}\mathcal{M}$ 
is finitely generated,
then  $\GKdim_{\mathrm{G}\mathcal{A}} \mathrm{G}\mathcal{M} = \GKdim_A M$.
\end{lemma}

\begin{proof}
By  arguments of Linear Algebra,
see \cite[Lemma~6.5 \& Proposition~6.6]{K-L}.
\end{proof}

\begin{theorem}\label{brodmann's challenge}
In the notation of the previous section,
it holds
\(
\Kdim_{K[X,Y]} \Gr^{\omega}M
=
\GKdim_{K[X,Y]} \Gr^{\omega}M
=
\GKdim_{W} M,
\)
and hence
\(
\Kdim\mathcal{V}^{\omega}(M)=\GKdim_{W}M,
\)
for all ${\omega\in\Omega}$.
\end{theorem}

\begin{proof}
Let $\omega\in\Omega$.
Since the $(1\vec 1)$-filtration of $K[X,Y]$ is finite,
any good $(1\vec 1)$-filtration of $\Gr^{\omega}M$ is finite and discrete
by \ref{finite and discrete}.
Thus by \ref{GM=M}, 
$
\smash{\GKdim_{K[X,Y]}\Gr^{\omega}M}=
\smash{\GKdim_{K[X,Y]}\Gr^{(1\vec 1)}\Gr^{\omega}M}
$.
By~\ref{good => fg}, $\Gr^{(1\vec 1)}\Gr^{\omega}M$ is finitely generated over $K[X,Y]$,
and so, by \ref{Kdim=GKdim 2}, 
$
\smash{\GKdim_{K[X,Y]}\Gr^{(1\vec 1)}\Gr^{\omega}M}=
\smash{\GKdim K[X,Y]\mod\rad{(0:\Gr^{(1\vec 1)}\Gr^{\omega}M)}}
$.
By \ref{main}, 
$
\smash{\GKdim K[X,Y]\mod\!\rad{(0:\Gr^{(1\vec 1)}\Gr^{\omega}M)}}
\!=\smash{\GKdim K[X,Y]\mod\rad{(0:\Gr^{(1\vec 1)+s\omega}M)}}
$,
${s\gg 0}$.
By \ref{Kdim=GKdim 2}, 
$
\smash{\GKdim K[X,Y]\mod\rad{(0:\Gr^{(1\vec 1)+s\omega}M)}}
=\smash{\GKdim_{K[X,Y]}\Gr^{(1\vec 1)+s\omega}M}
$,
$s\in\NN$.
Since  the ${(1\vec 1)+s\omega}$-filtrations of $W$ are finite, 
and therefore by \ref{finite and discrete} the good
$\smash{(1\vec 1)+s\omega}$-filtrations of $M$ are finite and discrete,
by \ref{GM=M} and \ref{iso} we obtain
$
\smash{\GKdim_{K[X,Y]}\Gr^{(1\vec 1)+s\omega}M}
=\smash{\GKdim_{W}M}
$,
${s\in\NN}$.
As for the Krull dimension, we conclude by~\ref{Kdim=GKdim 2}.
\end{proof}

%% file: boldini-cc_cv_wa-6.tex
\section{Application 2: Classification of Characteristic Varieties}\label{sec6}

\noindent
As before, let $K$ be a field of characteristic $0$.
For an arbitrary left ideal $L$ of the $1^{\text{st}}$ Weyl algebra $W$ over $K$
we aim to classify the characteristic varieties of~$W\mod L$.
More precisely,
we aim to partition $\Omega=\NN_{0}^{2}\wo\{(0,0)\}$ into regions corresponding to
equivalence classes $[\omega]_{\sim_L}$ of
weights $\omega\in\Omega$ such that $\omega'\sim_{L}\omega''$
if and only if
$\smash{\Gr^{\omega'}L=\Gr^{\omega''}L}$.
%
This would permit us to determine the number $\chi(L)$ of distinct ideals $\Gr^\omega L$,
${\omega\in\Omega}$, which we know to be finite by \ref{GL finite}.
Hence, because
$\smash{\Gr^{\omega'}L=\Gr^{\omega''}L}$
implies
$\smash{\mathcal{V}^{\omega'}(W\mod L)=\mathcal{V}^{\omega''}(W\mod L)}$ by \ref{annullatore},
$\chi(L)$ would be an upper bound for the number of distinct $\omega$-characteristic
varieties of $W\mod L$.

We do not succeed in this but by a computer experiment we approximate
 $\smash{\Omega\mod{\sim}_L}$ 
and this allows us to conjecture an upper bound for $\chi(L)$ in terms of total degrees of
universal Gr\"obner bases of $L$.

\begin{remark}\label{krites}
Let $n\in\NN$.
For each finitely generated left module $M$ over the $\nth$ Weyl algebra over $K$
and for each
${\nu\in\NN_0^{2n}}$  there exists a \emph{minimal} number ${\kappa_\nu(M)\in\NN_0}$
such that for all $\omega\in\Omega$ the characteristic varieties
$\mathcal{V}^{\sw}(M)$ stabilize to ${\Var(0:\Gr^\nu\Gr^\omega M)}$
as soon as $s>\kappa_\nu(M)$.

In particular,
$\smash{\mathcal{V}^{(1\vec 1)+s\omega}(M)}$ becomes precisely the critical cone
$\mathcal{C}^\omega(M)$ for all ${\omega\in\Omega}$ as soon as 
$s>\kappa(M)=\smash{\kappa_{(1\vec 1)}(M)}$.
\end{remark}

\begin{remark}\label{martinetto}
Let $n\in\NN$.
For each left ideal $L$ of the $\nth$ Weyl algebra over $K$ and for each $\nu\in\NN_0^{2n}$
we put $\igb_\nu(L)=\smash{\inf_{U}^{}\sup_{u\in U\wo\{0\}}\deg^\nu(u)}$, 
where the infimum is taken over all universal Gr\"obner bases $U$ of $L$.
By the proof of \ref{athena}, (a)~$\kappa_\nu(W\mod L)\le \igb_\nu(L)\in\NN_0$.
Clearly, (b)~$\igb_{\nu'}(L)\le\igb_{\nu''}(L)$ whenever $|\nu'|\le|\nu''|$.
Finally, (c)~$\igb_{k\nu}(L)=k\igb_\nu(L)$ for all $k\in\NN_0$.
\end{remark}

\begin{experiment}\label{6.3}
Let $L$ be any left ideal of the $1^{\text{st}}$ Weyl algebra $W$ over $K$.
By \ref{athena} we can compute an \emph{approximation} of
$\Omega\mod{\sim_L}$ if we know $\kappa_\nu(W\mod L)$
for all $\nu\in\NN_0^{2n}$.
By the relations (a), (b), (c) of \ref{martinetto} we have
$\kappa_\nu(W\mod L)\le\gamma_\nu(L)\le\gamma_{\vert\nu\vert(1\vec 1)}(L)=
\vert\nu\vert\gamma(L)$, where we put $\igb(L)=\igb_{(1\vec 1)}(L)$.
Therefore, by \ref{athena}, knowing
the upper bound
$\igb(L)$
of $\kappa(W\mod L)$
is sufficient for computing a (coarser) approximation
of~$\Omega\mod{\sim_L}$.

For some numbers $s_0\in\NN_0$ we repeatedly do an experiment parametrized by $s_0$ as follows.
A computer calculates for us the intersection points 
among the half-lines $\ell_{\nu,\omega}\subseteq\Omega$ of the form
$\ell_{\nu,\omega}(s)=\nu+s\omega$,
$\nu\in\NN_0^2$, $\omega\in\Omega$, for $s>s_0$,
and paints incident half-lines by a common colour.
The points of $\Omega$ having got the same colour turn out to build cones in $\Omega$.
For instance, for $s_0=3$ the computer program painted $17$ differently coloured cones, among 
which $9$ are degenerate, that is, half-lines.
For typographical reasons, in Figure~\ref{fig:classificazione} we depict 
the so obtained cones
 by connected regions in $\RR^2$,
alternately in black and gray.
For $s_0=3$ the $9$ degenerate cones are filled in black, whereas
the $8$ non-degenerate cones are filled in gray, and similarly in the other pictures of
Figure~\ref{fig:classificazione}.

\begin{figure}[ht] 
\centering
\setcounter{subfigure}{0}
\subfigure[$s_0=0$]{\scalebox{0.333}{\input{fan2.0.16x16.tex}}}
\quad
\subfigure[$s_0=1$]{\scalebox{0.333}{\input{fan2.1.16x16.tex}}}

\smallskip
\subfigure[$s_0=2$]{\scalebox{0.333}{\input{fan2.2.16x16.tex}}}
\quad
\subfigure[$s_0=3$]{\scalebox{0.333}{\input{fan2.3.16x16.tex}}}
\vskip -0.50\baselineskip
\caption{\sc Equality Regions of Characteristic Varieties}
\label{fig:classificazione}
\end{figure}
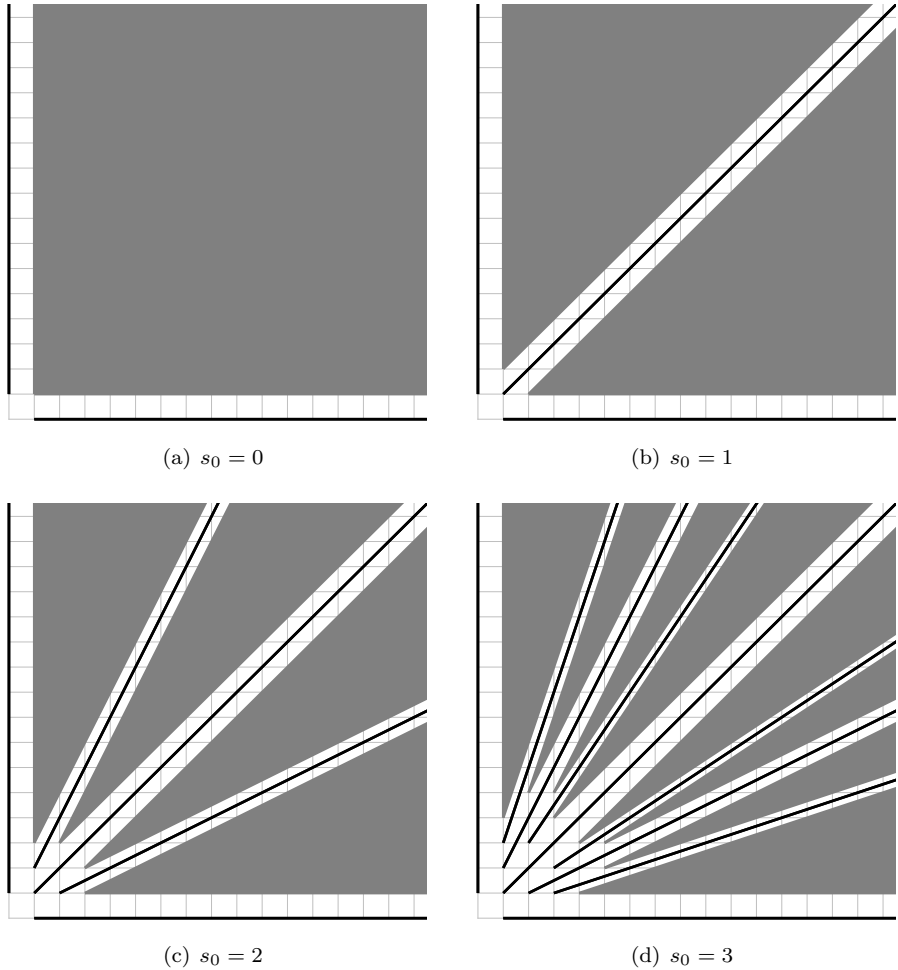

By \ref{athena}, as soon as $s_0\ge\gamma(L)$,
each of these cones is a subset of precisely one  equivalence class
of $\smash{\Omega\mod_{\sim_L}}$.
Thus the results of our experiment let us conjecture an upper bound
for
$\chi(L)$
in terms of $\gamma(L)$,
namely,
$\smash{\chi(L)\le 2^{1+\gamma(L)}+1}$.

Our experiment also indicates 
that the coordinates $(x_1,x_2)\in\NN_0^2$ of the vertices
of the cones lying in the lower semiquadrant without the diagonal
satisfy precisely the conditions
(a)~$F(1)\le x_1\le F(2+s_0)$,
(b)~$F(0)\le x_2\le F(1+s_0)$,
(c)~${\gcd(x_1,x_2)=1}$, 
and (d)~$x_1 > x_2$,
where $F(s)$ is the $\th{s}$ Fibonacci number,
that is, $F(0)=0$, $F(1)=1$, and $F(s)=F(s-1)+F(s-2)$ for all $s\ge 2$.
For instance, if ${s_0=3}$, 
these coordinates are $(1,0),\,(2,1),\,(3,1),\,(4,1),\,(3,2),\,(5,2),\,(4,3),\,(5,3)$.

So $2^{\gamma(L)}$ is equal to the number of the points $(x_1,x_2)\in\NN_0^2$
satisfying the conditions (a)--(d) with $s_0=\gamma(L)$,
and the experiment indicates that
\(
\chi(L)\le\#\{(x_{\varsigma(1)},x_{\varsigma(2)})\in\NN_0^2\mid
{\varsigma\in \Sigma_2}\mathop{\wedge} 
F(1)\!\le x_{1}\le F(2+\gamma(L))\mathop{\wedge} 
F(0)\!\le x_{2}\le F(1+\gamma(L))\)
\( \mathop{\wedge} 
{\gcd(x_1,x_2)=1}\mathop{\wedge} 
x_{1}\ge x_{2}\}
= \#\Sigma_2 \cdot (2^{\gamma(L)}+1)-(\#\Sigma_2-1)
= 2^{1+\gamma(L)}+1,
\)
where $\Sigma_2$ is the $2^{\text{nd}}$ symmetric group.
\end{experiment}

\begin{remark}
Weyl algebras are the prototype of algebras of solvable type, see \cite{Kan},
and as in the polynomial case a universal Gr\"obner basis of $L$ can be constructed
as a union of reduced Gr\"obner bases of $L$.
In \cite[Corollary~0.2]{Ley},
an upper bound is given for the total degree of elements of reduced
Gr\"obner bases of a left ideal of an algebra of solvable type in terms of the total degree
of generators of the ideal, thus in particular an upper bound for $\gamma(L)$.
Therefore if our conjecture is true,
one obtains an upper bound for the cardinality of $\Omega\mod{\sim_L}$
in terms of the total degree of generators of $L$.
\end{remark}

\begin{question} 
We may ask whether similar upper bounds for $\chi(L)$ as in \ref{6.3} exist
when considering a left ideal $L$ of the $\nth$ Weyl algebra for $n>1$, namely:
(1)~a bound in terms of $n$ and $\gamma(L)$,
%
and (2)~a bound in terms of Fibonacci numbers.
\end{question}

%% file: fan2.0.16x16.tex
%
\begin{pspicture*}(-0.5,-0.5)(16.5,16.5)
\psgrid[gridlabels=0,gridcolor=lightgray,subgriddiv=1](0,0)(17,17)
\psset{linewidth=3pt}
{\psset{linecolor=black}\qline(1,0)(18,0)\qline(1,0)(18,0)}
{\psset{linecolor=gray}\qline(1,1)(18,1)\qline(1,1)(1,18)}
{\psset{linecolor=gray,fillcolor=gray}\pswedge*(1,1){27}{0}{90}}
{\psset{linecolor=black}\qline(0,1)(0,18)\qline(0,1)(0,18)}
\end{pspicture*}
%

%% file: fan2.1.16x16.tex
%
\begin{pspicture*}(-0.5,-0.5)(16.5,16.5)
\psgrid[gridlabels=0,gridcolor=lightgray,subgriddiv=1](0,0)(17,17)
\psset{linewidth=3pt}
{\psset{linecolor=black}\qline(1,0)(18,0)\qline(1,0)(18,0)}
{\psset{linecolor=gray}\qline(2,1)(19,1)\qline(2,1)(19,18)}
{\psset{linecolor=gray,fillcolor=gray}\pswedge*(2,1){27}{0}{45}}
{\psset{linecolor=black}\qline(1,1)(18,18)\qline(1,1)(18,18)}
{\psset{linecolor=gray}\qline(1,2)(18,19)\qline(1,2)(1,19)}
{\psset{linecolor=gray,fillcolor=gray}\pswedge*(1,2){27}{45}{90}}
{\psset{linecolor=black}\qline(0,1)(0,18)\qline(0,1)(0,18)}
\end{pspicture*}
%

%% file: fan2.2.16x16.tex
%
\begin{pspicture*}(-0.5,-0.5)(16.5,16.5)
\psgrid[gridlabels=0,gridcolor=lightgray,subgriddiv=1](0,0)(17,17)
\psset{linewidth=3pt}
{\psset{linecolor=black}\qline(1,0)(18,0)\qline(1,0)(18,0)}
{\psset{linecolor=gray}\qline(3,1)(20,1)\qline(3,1)(21,10)}
{\psset{linecolor=gray,fillcolor=gray}\pswedge*(3,1){27}{0}{26.565}}
{\psset{linecolor=black}\qline(2,1)(20,10)\qline(2,1)(20,10)}
{\psset{linecolor=gray}\qline(3,2)(21,11)\qline(3,2)(20,19)}
{\psset{linecolor=gray,fillcolor=gray}\pswedge*(3,2){27}{26.565}{45}}
{\psset{linecolor=black}\qline(1,1)(18,18)\qline(1,1)(18,18)}
{\psset{linecolor=gray}\qline(2,3)(19,20)\qline(2,3)(11,21)}
{\psset{linecolor=gray,fillcolor=gray}\pswedge*(2,3){27}{45}{63.435}}
{\psset{linecolor=black}\qline(1,2)(10,20)\qline(1,2)(10,20)}
{\psset{linecolor=gray}\qline(1,3)(10,21)\qline(1,3)(1,20)}
{\psset{linecolor=gray,fillcolor=gray}\pswedge*(1,3){27}{63.435}{90}}
{\psset{linecolor=black}\qline(0,1)(0,18)\qline(0,1)(0,18)}
\end{pspicture*}
%

%% file: fan2.3.16x16.tex
%
\begin{pspicture*}(-0.5,-0.5)(16.5,16.5)
\psgrid[gridlabels=0,gridcolor=lightgray,subgriddiv=1](0,0)(17,17)
\psset{linewidth=3pt}
{\psset{linecolor=black}\qline(1,0)(18,0)\qline(1,0)(18,0)}
{\psset{linecolor=gray}\qline(4,1)(21,1)\qline(4,1)(22,7)}
{\psset{linecolor=gray,fillcolor=gray}\pswedge*(4,1){27}{0}{18.435}}
{\psset{linecolor=black}\qline(3,1)(21,7)\qline(3,1)(21,7)}
{\psset{linecolor=gray}\qline(5,2)(23,8)\qline(5,2)(23,11)}
{\psset{linecolor=gray,fillcolor=gray}\pswedge*(5,2){27}{18.435}{26.565}}
{\psset{linecolor=black}\qline(2,1)(20,10)\qline(2,1)(20,10)}
{\psset{linecolor=gray}\qline(5,3)(23,12)\qline(5,3)(23,15)}
{\psset{linecolor=gray,fillcolor=gray}\pswedge*(5,3){27}{26.565}{33.690}}
{\psset{linecolor=black}\qline(3,2)(21,14)\qline(3,2)(21,14)}
{\psset{linecolor=gray}\qline(4,3)(22,15)\qline(4,3)(21,20)}
{\psset{linecolor=gray,fillcolor=gray}\pswedge*(4,3){27}{33.690}{45}}
{\psset{linecolor=black}\qline(1,1)(18,18)\qline(1,1)(18,18)}
{\psset{linecolor=gray}\qline(3,4)(20,21)\qline(3,4)(15,22)}
{\psset{linecolor=gray,fillcolor=gray}\pswedge*(3,4){27}{45}{56.310}}
{\psset{linecolor=black}\qline(2,3)(14,21)\qline(2,3)(14,21)}
{\psset{linecolor=gray}\qline(3,5)(15,23)\qline(3,5)(12,23)}
{\psset{linecolor=gray,fillcolor=gray}\pswedge*(3,5){27}{56.310}{63.435}}
{\psset{linecolor=black}\qline(1,2)(10,20)\qline(1,2)(10,20)}
{\psset{linecolor=gray}\qline(2,5)(11,23)\qline(2,5)(8,23)}
{\psset{linecolor=gray,fillcolor=gray}\pswedge*(2,5){27}{63.435}{71.565}}
{\psset{linecolor=black}\qline(1,3)(7,21)\qline(1,3)(7,21)}
{\psset{linecolor=gray}\qline(1,4)(7,22)\qline(1,4)(1,21)}
{\psset{linecolor=gray,fillcolor=gray}\pswedge*(1,4){27}{71.565}{90}}
{\psset{linecolor=black}\qline(0,1)(0,18)\qline(0,1)(0,18)}
\end{pspicture*}
%